\documentclass[10pt,reqno]{amsart}
\usepackage[hypertex]{hyperref}

\usepackage{amsmath,amsthm}
\usepackage{amssymb}
\usepackage{euscript}
\usepackage[all,tips]{xy}

\numberwithin{equation}{subsection}

\newcommand{\sqsp}{\renewcommand{\baselinestretch}{1.15}\tiny\normalsize}

\newcommand{\relphantom[1]}{\mathrel{\phantom{#1}}}

\newtheorem{theorem}[subsection]{Theorem}
\newtheorem{lemma}[subsection]{Lemma}
\newtheorem{proposition}[subsection]{Proposition}
\newtheorem{corollary}[subsection]{Corollary}

\theoremstyle{definition}
\newtheorem{definition}[subsection]{Definition}
\newtheorem{example}[subsection]{Example}
\newtheorem{remark}[subsection]{Remark}

\newcommand{\bk}{\mathbf{k}}
\newcommand{\oct}{\mathbf{O}}
\newcommand{\octalpha}{\mathbf{O}_\alpha}
\newcommand{\octalphaminus}{\mathbf{O}_\alpha^-}
\newcommand{\octminus}{\mathbf{O}^-}

\newcommand{\xbar}{\overline{x}}
\newcommand{\ybar}{\overline{y}}
\newcommand{\zbar}{\overline{z}}

\newcommand{\mun}{\mu^{(n)}}

\newcommand{\muone}{\mu^{(1)}}
\newcommand{\mualpha}{\mu_\alpha}

\newcommand{\cyclicsum}{{\circlearrowright\,}}

\DeclareMathOperator{\Hom}{Hom}

\newcommand{\matrixx}{{\begin{pmatrix}
1 & 0 & e_1\\
0 & 1 & 0\\
-e_1 & 0 & 1\end{pmatrix}}}

\newcommand{\matrixy}{{\begin{pmatrix}
1 & e_2 & e_3\\
-e_2 & 0 & 0\\
-e_3 & 0 & 0\end{pmatrix}}}

\newcommand{\matrixz}{{\begin{pmatrix}
2 & \frac{1}{2}e_3 + \frac{3}{2}e_6 & e_1 + e_2 + 2e_7\\
-\frac{1}{2}e_3 - \frac{3}{2}e_6 & 0 & -\frac{3}{4}e_5 + e_7\\
-e_1 - e_2 - 2e_7 & \frac{3}{4}e_5 - e_7 & 0\end{pmatrix}}}

\newcommand{\matrixw}{{\begin{pmatrix}
3 & 2e_6 & 2e_5 + 2e_7\\
-2e_6 & 0 & \frac{7}{4}e_1\\
-2e_5 - 2e_7 & -\frac{7}{4}e_1 & 1\end{pmatrix}}}

\begin{document}

\title{Hom-Maltsev, Hom-alternative, and Hom-Jordan algebras}
\author{Donald Yau}

\begin{abstract}
Hom-Maltsev(-admissible) algebras are defined, and it is shown that Hom-alternative algebras are Hom-Maltsev-admissible.  With a new definition of a Hom-Jordan algebra, it is shown that Hom-alternative algebras are Hom-Jordan-admissible.  Hom-type generalizations of some well-known identities in alternative algebras, including the Moufang identities, are obtained.
\end{abstract}



\address{Department of Mathematics\\
    The Ohio State University at Newark\\
    1179 University Drive\\
    Newark, OH 43055, USA}
\email{dyau@math.ohio-state.edu}

\date{\today}
\maketitle

\sqsp

\section{Introduction}

Maltsev algebras were introduced by Maltsev \cite{maltsev}, who called these objects Moufang-Lie algebras.  A Maltsev algebra is a non-associative algebra $A$ with an anti-symmetric multiplication $[-,-]$ that satisfies the Maltsev identity
\begin{equation}
\label{maltsev}
J(x,y,[x,z]) = [J(x,y,z),x]
\end{equation}
for all $x,y,z \in A$, where $J(x,y,z) = [[x,y],z] + [[z,x],y] + [[y,z],x]$ is the Jacobian.  In particular, Lie algebras are examples of Maltsev algebras.   Maltsev algebras play an important role in the geometry of smooth loops.   Just as the tangent algebra of a Lie group is a Lie algebra, the tangent algebra of a locally analytic Moufang loop is a Maltsev algebra \cite{kerdman,kuzmin,maltsev,nagy,sabinin}.  The reader is referred to \cite{gt,myung,okubo} for discussions about the relationships between Maltsev algebras, exceptional Lie algebras, and physics.

Closely related to Maltsev algebras are alternative algebras.  An alternative algebra is an algebra whose associator is an alternating function.  In particular, all associative algebras are alternative, but there are plenty of non-associative alternative algebras, such as the octonions.  Roughly speaking, alternative algebras are related to Maltsev algebras as associative algebras are related to Lie algebras.  Indeed, as Maltsev observed in \cite{maltsev}, every alternative algebra $A$ is Maltsev-admissible, i.e., the commutator algebra $A^-$ is a Maltsev algebra.  There are many Maltsev-admissible algebras that are not alternative; see, e.g., \cite{myung}.  The reader is referred to \cite{tw} for applications of alternative algebras to projective geometry, buildings, and algebraic groups.

Instead of the commutator, the anti-commutator also gives rise to interesting structures.  A Jordan algebra is a commutative algebra that satisfies the Jordan identity
\begin{equation}
\label{jordan}
(x^2y)x = x^2(yx).
\end{equation}
Starting with an alternative algebra $A$, it is known that the Jordan product
\[
x \ast y = \frac{1}{2}(xy + yx)
\]
gives a Jordan algebra $A^+ = (A,\ast)$.  In other words, alternative algebras are Jordan-admissible.  The reader is referred to \cite{gt,jvw,okubo,sv} for discussions about the important roles of Jordan algebras in physics, especially quantum mechanics.

The purpose of this paper is to study Hom-type generalizations of Maltsev(-admissible) algebras, alternative algebras, and Jordan(-admissible) algebras.  The reader is referred to the survey article \cite{mak2} for discussions about other Hom-type algebras and to \cite{yau3}-\cite{yau11} for Hom-type analogues of Novikov algebras, quantum groups, and the Yang-Baxter equations.  Roughly speaking, a Hom-type generalization of a kind of algebras is defined by twisting the defining identities by a self-map, called the twisting map.  When the twisting map is the identity map, one recovers the original kind of algebras.

Below is a description of the rest of this paper.

In section \ref{sec:hommaltsev} we define Hom-Maltsev algebras and prove two construction results, Theorems \ref{thm:maltsevtp2} and \ref{thm:maltsevtp1}.  Hom-Maltsev algebras include Maltsev algebras and Hom-Lie algebras as examples.  Theorem \ref{thm:maltsevtp2} says that the class of Hom-Maltsev algebras is closed under the process of taking derived Hom-algebras (Definition \ref{def:derivedhomalge}), in which the structure maps are suitably twisted by the twisting map.  Theorem \ref{thm:maltsevtp1} says that a Maltsev algebra $(A,[-,-])$ can be twisted into a Hom-Maltsev algebra $A_\alpha = (A,[-,-]_\alpha=\alpha\circ[-,-],\alpha)$ along any algebra self-map $\alpha$ of $A$.  In Examples \ref{ex:4dmaltsev} and \ref{ex:5dmatlsev}, we show that, using Theorem \ref{thm:maltsevtp1} with different algebra self-maps, it is possible to twist a Maltsev algebra into a non-Hom-Lie Hom-Maltsev algebra, a Hom-Lie algebra, or a Lie algebra.

The Hom-type analogue of an alternative algebra is called a Hom-alternative algebra, in which the Hom-associator \eqref{homassociator} is alternating.  Hom-alternative algebras were introduced by Makhlouf in \cite{mak}.  In section \ref{sec:homalternative} we show that Hom-alternative algebras are Hom-Maltsev-admissible (Theorem \ref{thm:homaltmaltsev}).  That is, the commutator Hom-algebra (Definition \ref{def:commutatoralgebra}) of a Hom-alternative algebra is a Hom-Maltsev algebra, generalizing the fact that alternative algebras are Maltsev-admissible.  Hom-Lie-admissible algebras \cite{ms} and Maltsev-admissible algebras are obvious examples of Hom-Maltsev-admissible algebras.  The proof of the Hom-Maltsev-admissibility of Hom-alternative algebras involves the Hom-type analogues of certain identities that hold in alternative algebras and of the Bruck-Kleinfeld function (Definition \ref{def:bk}).  In Example \ref{ex:oct}, starting with the octonions, we construct (non-Hom-Lie) Hom-Maltsev algebras using Theorem \ref{thm:homaltmaltsev}.

In section \ref{sec:hommaltsevad} we consider the class of Hom-Maltsev-admissible algebras.  In Proposition \ref{prop:hmad} we give several characterizations of Hom-Maltsev-admissible algebras that are also Hom-flexible \cite{ms}.  In Theorems \ref{thm:homf} and \ref{thm:hommalad} we prove construction results for Hom-flexible and Hom-Maltsev-admissible algebras.  Hom-alternative algebras are Hom-flexible \cite{mak}, so by Theorem \ref{thm:homaltmaltsev} Hom-alternative algebras are both Hom-flexible and Hom-Maltsev-admissible.  In Examples \ref{ex:hma5}, \ref{ex:hma6}, and \ref{ex:hma8}, we construct Hom-flexible, Hom-Maltsev-admissible algebras that are not Hom-alternative, not Hom-Lie-admissible, and not Maltsev-admissible.

In section \ref{sec:jordan} we study Hom-Jordan(-admissible) algebras, which are the Hom-type generalizations of Jordan(-admissible) algebras.  The first definition of a Hom-Jordan algebra was given by Makhlouf in \cite{mak}.  Hom-alternative algebras are not Hom-Jordan-admissible under that definition.  We introduce a different definition of a Hom-Jordan algebra and show that Hom-alternative algebras are Hom-Jordan-admissible under this new definition (Theorem \ref{thm:hahj}).  In other words, the plus Hom-algebra (Definition \ref{def:plushom}) of any Hom-alternative algebra is a Hom-Jordan algebra, generalizing the Jordan-admissibility of alternative algebras.  Construction results analogous to Theorems \ref{thm:maltsevtp2} and \ref{thm:maltsevtp1} are proved for Hom-Jordan(-admissible) algebras (Theorems \ref{thm:hjtp} and \ref{thm:hjatp}).  In Example \ref{ex:m83}, we construct (non-Jordan) Hom-Jordan algebras using the $27$-dimensional exceptional simple Jordan algebra of $3 \times 3$ Hermitian octonionic matrices.

In section \ref{sec:moufang} we provide further properties for Hom-alternative algebras.  First we observe that a Hom-algebra is Hom-associative if and only if it is both Hom-alternative and Hom-Lie-admissible (Proposition \ref{cor:homaltass}).  Then we show that the class of Hom-alternative algebras is closed under taking derived Hom-algebras (Proposition \ref{prop:alttp2}).  In Propositions \ref{prop:f2} and \ref{prop:g} we provide further properties of the Hom-Bruck-Kleinfeld function in Hom-alternative algebras.  It is well-known that the Moufang identities \eqref{moufang} hold in alternative algebras.  In Theorem \ref{thm:moufang} we show that there are Hom-type generalizations of the Moufang identities in Hom-alternative algebras.

\section{Hom-Maltsev algebras}
\label{sec:hommaltsev}

In this section we define Hom-Maltsev algebras and study their general properties.  Other characterizations of the Hom-Maltsev identity are given (Proposition \ref{prop:hommaltsevid}).  We prove some construction results for Hom-Maltsev algebras (Theorems \ref{thm:maltsevtp2} and \ref{thm:maltsevtp1}).  Using Theorem \ref{thm:maltsevtp1}, we demonstrate in Examples \ref{ex:4dmaltsev} and \ref{ex:5dmatlsev} that it is possible to twist a Maltsev algebra into a non-Hom-Lie Hom-Maltsev algebra or a Hom-Lie (or even Lie) algebra using different algebra morphisms.

\subsection{Conventions}
Throughout the rest of this paper, we work over a fixed field $\bk$ of characteristic $0$.  Modules, tensor products, linearity, and $\Hom$ are all meant over $\bk$. If $f \colon V \to V$ is a linear self-map on a vector space $V$, then $f^n \colon V \to V$ denotes the composition $f \circ \cdots \circ f$ of $n$ copies of $f$, with $f^0 = Id$.  For a map $\mu \colon V^{\otimes 2} \to V$, we sometimes write $\mu(a,b)$ as $ab$ for $a,b \in V$.  If $W$ is another vector space, then $\tau \colon V \otimes W \cong W \otimes V$ denotes the twist isomorphism, $\tau(v \otimes w) = w \otimes v$.  More generally, we do not distinguish between a permutation $\theta$ on $n$ letters and its induced linear isomorphism $\theta \colon V^{\otimes n} \to V^{\otimes n}$ given by $\theta(v_1 \otimes \cdots \otimes v_n) = v_{\theta(1)} \otimes \cdots\otimes v_{\theta(n)}$.


Let us give the definitions regarding Hom-algebras.

\begin{definition}
\label{def:homalgebra}
By a \textbf{Hom-algebra} we mean a triple $(A,\mu,\alpha)$ in which $A$ is a $\bk$-module, $\mu \colon A^{\otimes 2} \to A$ is a bilinear map (the multiplication), and $\alpha \colon A \to A$ is a linear map (the twisting map) such that $\alpha\circ\mu = \mu\circ\alpha^{\otimes 2}$ (multiplicativity).  A Hom-algebra $(A,\mu,\alpha)$ is usually denoted simply by $A$.  A \textbf{morphism} $f \colon A \to B$ of Hom-algebras is a linear map $f$ of the underlying $\bk$-modules such that $f\circ\alpha_A = \alpha_B \circ f$ and $\mu_B\circ f^{\otimes 2} = f \circ \mu_A$.
\end{definition}

\begin{remark}
\label{rk:classicalalgebra}
If $(A,\mu)$ is a not-necessarily associative algebra in the usual sense, we also regard it as the Hom-algebra $(A,\mu,Id)$ with identity twisting map.  This defines a fully faithful embedding from the category of algebras into the category of Hom-algebras.  With this convention, the notion of a morphism between Hom-algebras with identity twisting maps reduces to the usual notion of an algebra morphism.
\end{remark}

\begin{remark}
The multiplicativity of the twisting map $\alpha$ is built into our definition of a Hom-algebra.  Some authors (see, e.g., \cite{mak,mak2,ms}) do not make this assumption.  We chose to impose multiplicativity because many of our results depend on it and all of our concrete examples of Hom-Maltsev(-admissible), Hom-alternative, and Hom-Jordan(-admissible) algebras have this property.
\end{remark}

The algebraic structures studied in this paper are all defined using the Hom-versions of the associator and the Jacobian, which we now define.

\begin{definition}
\label{def:homassociator}
Let $(A,\mu,\alpha)$ be a Hom-algebra.
\begin{enumerate}
\item
The \textbf{Hom-associator} of $A$ \cite{ms} is the trilinear map $as_A \colon A^{\otimes 3} \to A$ defined as
\begin{equation}
\label{homassociator}
as_A = \mu \circ (\mu \otimes \alpha - \alpha \otimes \mu).
\end{equation}
\item
The \textbf{Hom-Jacobian} of $A$ \cite{ms} is the trilinear map $J_A \colon A^{\otimes 3} \to A$ defined as
\begin{equation}
\label{homjacobian}
J_A = \mu \circ (\mu \otimes \alpha) \circ (Id + \sigma + \sigma^2),
\end{equation}
where $\sigma \colon A^{\otimes 3} \to A^{\otimes 3}$ is the cyclic permutation $\sigma(x \otimes y \otimes z) = z \otimes x \otimes y$.
\end{enumerate}
If there is only one Hom-algebra under consideration, we will sometimes omit the subscript in the Hom-associator and the Hom-Jacobian
\end{definition}

Note that when $(A,\mu)$ is an algebra (with $\alpha = Id$), its Hom-associator and Hom-Jacobian coincide with its usual associator and Jacobian, respectively.

Since Hom-Maltsev algebras generalize Hom-Lie algebras (as we will see shortly), which in turn generalize Lie algebras, we use the bracket notation $[-,-]$ to denote their multiplications.

\begin{definition}
\label{def:hommaltsev}
\begin{enumerate}
\item
A \textbf{Hom-Lie algebra} \cite{hls,ms} is a Hom-algebra $(A,[-,-],\alpha)$ such that $[-,-]$ is anti-symmetric (i.e., $[-,-]\circ(Id + \tau) = 0$) and that the \textbf{Hom-Jacobi identity}
\begin{equation}
\label{homjacobiid}
J_A = 0
\end{equation}
is satisfied, where $J_A$ is the Hom-Jacobian of $A$ \eqref{homjacobian}.
\item
A \textbf{Hom-Maltsev algebra} is a Hom-algebra $(A,[-,-],\alpha)$ such that $[-,-]$ is anti-symmetric and that the \textbf{Hom-Maltsev identity}
\begin{equation}
\label{hommaltsevid}
J_A(\alpha(x),\alpha(y),[x,z]) = [J_A(x,y,z),\alpha^2(x)]
\end{equation}
is satisfied for all $x,y,z \in A$.
\end{enumerate}
\end{definition}

Observe that when $\alpha = Id$, the Hom-Jacobi identity reduces to the usual Jacobi identity
\[
[[x,y],z] + [[z,x],y] + [[y,z],x] = 0
\]
for all $x,y,z \in A$.  Likewise, when $\alpha = Id$, by the anti-symmetry of $[-,-]$, the Hom-Maltsev identity reduces to the Maltsev identity \eqref{maltsev} or equivalently,
\begin{equation}
\label{maltsevidentity}
[[x,y],[x,z]] = [[[x,y],z],x] + [[[y,z],x],x] + [[[z,x],x],y]
\end{equation}
for all $x,y,z \in A$.

\begin{example}
\label{ex:homlie}
A Lie (resp., Maltsev \cite{maltsev}) algebra $(A,[-,-])$ is a Hom-Lie (resp., Hom-Maltsev) algebra with $\alpha = Id$, since the Hom-Jacobi identity \eqref{homjacobiid} (resp., the Hom-Maltsev identity \eqref{hommaltsevid}) reduces to the usual Jacobi (resp., Maltsev) identity when $\alpha = Id$.  Moreover, every Hom-Lie algebra is also a Hom-Maltsev algebra because the Hom-Jacobi identity $J_A = 0$ clearly implies the Hom-Maltsev identity.\qed
\end{example}

Before we give more examples of Hom-Maltsev algebras, let us give some other characterizations of the Hom-Maltsev identity.

\begin{proposition}
\label{prop:hommaltsevid}
Let $(A,[-,-],\alpha)$ be a Hom-algebra with $[-,-]$ anti-symmetric.  Then the following statements are equivalent.
\begin{enumerate}
\item
$A$ is a Hom-Maltsev algebra, i.e., the Hom-Maltsev identity \eqref{hommaltsevid} holds.
\item
The condition
\begin{equation}
\label{hommaltsevid2}
\begin{split}
J(\alpha(w),\alpha(y),[x,z]) &+ J(\alpha(x),\alpha(y),[w,z])\\
&= [J(w,y,z),\alpha^2(x)] + [J(x,y,z),\alpha^2(w)]
\end{split}
\end{equation}
holds for all $w,x,y,z \in A$.
\item
The condition
\begin{equation}
\label{hommaltsevid3}
\begin{split}
\alpha([[x,y],[x,z]]) &= [[[x,y],\alpha(z)],\alpha^2(x)] + [[[y,z],\alpha(x)],\alpha^2(x)]\\
&\relphantom{} + [[[z,x],\alpha(x)],\alpha^2(y)]
\end{split}
\end{equation}
holds for all $x,y,z \in A$.
\item
The condition
\begin{equation}
\label{hommaltsevid4}
\begin{split}
\alpha([[w,y],[x,z]]) & + \alpha([[x,y],[w,z]])\\
&= [[[w,y],\alpha(z)],\alpha^2(x)] + [[[x,y],\alpha(z)],\alpha^2(w)]\\
&\relphantom{} + [[[y,z],\alpha(w)],\alpha^2(x)] + [[[y,z],\alpha(x)],\alpha^2(w)]\\
&\relphantom{} + [[[z,w],\alpha(x)],\alpha^2(y)] + [[[z,x],\alpha(w)],\alpha^2(y)]
\end{split}
\end{equation}
holds for all $w,x,y,z \in A$.
\end{enumerate}
\end{proposition}

\begin{proof}
The equivalence between the Hom-Maltsev identity \eqref{hommaltsevid} and \eqref{hommaltsevid2} follows from linearization: To get the latter, replace $x$ with $w + x$ in the former.  Conversely, \eqref{hommaltsevid2} yields \eqref{hommaltsevid} by setting $w = x$.  Similarly, linearization implies the equivalence between \eqref{hommaltsevid3} and \eqref{hommaltsevid4}.

To prove the equivalence between the Hom-Maltsev identity and \eqref{hommaltsevid3}, observe that the left-hand side of the Hom-Maltsev identity \eqref{hommaltsevid} is:
\[
\begin{split}
J(\alpha(x),\alpha(y),[x,z])
&= [[\alpha(x),\alpha(y)],\alpha([x,z])] + [[[x,z],\alpha(x)],\alpha^2(y)] + [[\alpha(y),[x,z]],\alpha^2(x)]\\
&= \alpha([[x,y],[x,z]]) - [[[z,x],\alpha(x)],\alpha^2(y)] + [[[z,x],\alpha(y)],\alpha^2(x)].
\end{split}
\]
In the last equality above, we used the multiplicativity of $\alpha$ and the anti-symmetry of $[-,-]$.  Likewise, the right-hand side of the Hom-Maltsev identity \eqref{hommaltsevid} is:
\[
[J(x,y,z),\alpha^2(x)]
= [[[x,y],\alpha(z)],\alpha^2(x)] + [[[z,x],\alpha(y)],\alpha^2(x)] + [[[y,z],\alpha(x)],\alpha^2(x)].
\]
Since the summand $[[[z,x],\alpha(y)],\alpha^2(x)]$ appears on both sides of \eqref{hommaltsevid}, the above calculation and a rearrangement of terms imply the equivalence between the Hom-Maltsev identity and \eqref{hommaltsevid3}.
\end{proof}

To state our next result, we need the following definition.

\begin{definition}
\label{def:derivedhomalge}
Let $(A,\mu,\alpha)$ be a Hom-algebra and $n \geq 0$.  Define the \textbf{$n$th derived Hom-algebra} of $A$ by
\[
A^n = (A,\mun = \alpha^{2^n-1}\circ\mu,\alpha^{2^n}).
\]
Note that $A^0 = A$, $A^1 = (A,\muone = \alpha\circ\mu, \alpha^2)$, and $A^{n+1} = (A^n)^1$.
\end{definition}

The following elementary observations are used in the next result.

\begin{lemma}
\label{lem:jalpha}
Let $(A,\mu,\alpha)$ be a Hom-algebra.  Then we have
\begin{equation}
\label{jalpha}
J_A \circ \alpha^{\otimes 3} = \alpha \circ J_A
\end{equation}
and
\begin{equation}
\label{jan}
J_{A^n} = \alpha^{2(2^n-1)} \circ J_A
\end{equation}
for all $n \geq 0$
\end{lemma}

\begin{proof}
The condition \eqref{jalpha} holds because $\alpha^{\otimes 3}$ commutes with the cyclic permutation $\sigma$ and $\alpha$ is multiplicative.  For \eqref{jan}, observe that:
\begin{equation}
\label{munalpha2n}
\begin{split}
\mun \circ (\mun \otimes \alpha^{2^n})
&= \alpha^{2^n-1} \circ \mu \circ ((\alpha^{2^n-1} \circ \mu) \otimes \alpha^{2^n})\\
&= \alpha^{2^n-1} \circ \mu \circ (\alpha^{2^n-1})^{\otimes 2} \circ (\mu \otimes \alpha)\\
&= \alpha^{2(2^n-1)} \circ \mu \circ (\mu \otimes \alpha),
\end{split}
\end{equation}
where the last equality follows from the multiplicativity of $\alpha$ with respect to $\mu$.  Pre-composing \eqref{munalpha2n} with the cyclic sum $(Id + \sigma + \sigma^2)$, we obtain \eqref{jan}.
\end{proof}

The following result shows that the category of Hom-Maltsev algebras is closed under taking derived Hom-algebras.

\begin{theorem}
\label{thm:maltsevtp2}
Let $(A,[-,-],\alpha)$ be a Hom-Maltsev algebra.  Then the $n$th derived Hom-algebra
\[
A^n = (A,[-,-]^{(n)} = \alpha^{2^n-1} \circ [-,-],\alpha^{2^n})
\]
is also a Hom-Maltsev algebra for each $n \geq 0$.
\end{theorem}

\begin{proof}
Since $A^0 = A$, $A^1 = (A,[-,-]^{(1)} = \alpha \circ [-,-],\alpha^2)$, and $A^{n+1} = (A^n)^1$, by an induction argument it suffices to prove the case $n = 1$.

To show that $A^1$ is a Hom-Maltsev algebra, first note that $[-,-]^{(1)}$ is anti-symmetric because $[-,-]$ is anti-symmetric and $\alpha$ is linear.  Since $\alpha^2$ is multiplicative with respect to $[-,-]^{(1)}$, it remains to show the Hom-Maltsev identity \eqref{hommaltsevid} for $A^1$.  For $x,y,z \in A$, we compute as follows:
\[
\begin{split}
J_{A^1}(\alpha^2(x),\alpha^2(y),[x,z]^{(1)})
&= J_{A^1}(\alpha^2(x),\alpha^2(y),\alpha([x,z]))\\
&= \alpha^2\left(J_A(\alpha^2(x),\alpha^2(y),\alpha([x,z]))\right)\quad\text{(by \eqref{jan})}\\
&= \alpha^3\left(J_A(\alpha(x),\alpha(y),[x,z])\right)\quad\text{(by \eqref{jalpha})}\\
&= \alpha^3 \left([J_A(x,y,z),\alpha^2(x)]\right)\quad\text{(by \eqref{hommaltsevid} in $A$)}\\
&= [\alpha^2(J_A(x,y,z)),(\alpha^2)^2(x)]^{(1)}\quad\text{(by multiplicativity of $\alpha$)}\\
&= [J_{A^1}(x,y,z),(\alpha^2)^2(x)]^{(1)}\quad\text{(by \eqref{jan})}.\\
\end{split}
\]
This establishes the Hom-Maltsev identity in $A^1$ and finishes the proof.
\end{proof}

To state our next result, we need the following definition.

\begin{definition}
\label{def:aalpha}
Let $(A,\mu)$ be any algebra and $\alpha \colon A \to A$ be an algebra morphism.  Define the \textbf{Hom-algebra induced by $\alpha$} as
\[
A_\alpha = (A,\mualpha,\alpha),
\]
where $\mualpha = \alpha\circ\mu$.
\end{definition}

The next result shows that given a Maltsev algebra and an algebra morphism, the induced Hom-algebra is a Hom-Maltsev algebra.  The Hom-Maltsev algebras constructed using this twisting result are generally not Maltsev algebras, as we will see in the examples later.  Such a twisting result was first used by the author in \cite{yau2} (Theorem 2.3) on $G$-associative algebras (where $G$ is a subgroup of the symmetric group on three letters), which include associative, Lie, pre-Lie, and Lie-admissible algebras as examples.  That result has since been employed and extended by various authors; see \cite{ama} (Theorem 2.7), \cite{ams} (Theorems 1.7 and 2.6), \cite{fg2} (Section 2), \cite{gohr} (Proposition 1), \cite{mak} (Theorems 2.1 and 3.5), \cite{mak2}, and \cite{ms4}.

\begin{theorem}
\label{thm:maltsevtp1}
Let $(A,[-,-])$ be a Maltsev algebra and $\alpha \colon A \to A$ be an algebra morphism.  Then the Hom-algebra $A_\alpha = (A,[-,-]_\alpha = \alpha\circ[-,-],\alpha)$ induced by $\alpha$ is a Hom-Maltsev algebra.
\end{theorem}

\begin{proof}
Since $[-,-]_\alpha$ is clearly anti-symmetric, it remains to prove the Hom-Maltsev identity \eqref{hommaltsevid} for $A_\alpha$.  Here we regard $A$ as the Hom-Maltsev algebra $(A,[-,-],Id)$ with identity twisting map.  For \emph{any} algebra $(A,[-,-])$ (Maltsev or otherwise), by the multiplicativity of $\alpha$ with respect to $[-,-]$, we have
\[
[-,-]_\alpha\circ([-,-]_\alpha \otimes \alpha) = \alpha^2 \circ[-,-]\circ([-,-]\otimes Id)
\]
and
\[
[-,-]\circ([-,-]\otimes Id)\circ\alpha^{\otimes 3} = \alpha\circ[-,-]\circ([-,-] \otimes Id).
\]
Pre-composing these identities with the cyclic sum $(Id + \sigma + \sigma^2)$ \eqref{homjacobian} (which commutes with $\alpha^{\otimes 3}$), we obtain
\begin{equation}
\label{jaalphaj}
J_{A_\alpha} = \alpha^2 \circ J_A
\end{equation}
and
\begin{equation}
\label{alphajj}
J_A \circ \alpha^{\otimes 3} = \alpha \circ J_A.
\end{equation}
To prove the Hom-Maltsev identity for $A_\alpha$, we compute as follows:
\[
\begin{split}
J_{A_\alpha}(\alpha(x),\alpha(y),[x,z]_\alpha)
&= \alpha^2\left(J_A(\alpha(x),\alpha(y),[\alpha(x),\alpha(z)])\right) \quad\text{(by \eqref{jaalphaj})}\\
&= \alpha^2\left([J_A(\alpha(x),\alpha(y),\alpha(z)),\alpha(x)]\right)\quad\text{(by \eqref{hommaltsevid} in $A$)}\\
&= [\alpha(J_A(\alpha(x),\alpha(y),\alpha(z)),\alpha^2(x)]_\alpha\quad\text{(by multiplicativity of $\alpha$)}\\
&= [\alpha^2(J_A(x,y,z)),\alpha^2(x)]_\alpha\quad\text{(by \eqref{alphajj})}\\
&= [J_{A_\alpha}(x,y,z),\alpha^2(x)]_\alpha\quad\text{(by \eqref{jaalphaj})}.
\end{split}
\]
This shows that the Hom-Maltsev identity holds in $A_\alpha$.
\end{proof}

We now discuss examples of Hom-Maltsev algebras that can be constructed using Theorem \ref{thm:maltsevtp1}.

\begin{example}
\label{ex:4dmaltsev}
There is a four-dimensional non-Lie Maltsev algebra $(A,[-,-])$ \cite{sagle} (Example 3.1) with basis $\{e_1,e_2,e_3,e_4\}$ and multiplication table:
\begin{center}
\begin{tabular}{c|cccc}
$[-,-]$ & $e_1$ & $e_2$ & $e_3$ & $e_4$ \\ \hline
$e_1$ & $0$ & $-e_2$ & $-e_3$ & $e_4$ \\
$e_2$ & $e_2$ & $0$ & $2e_4$ & $0$ \\
$e_3$ & $e_3$ & $-2e_4$ & $0$ & $0$ \\
$e_4$ & $-e_4$ & $0$ & $0$ & $0$ \\
\end{tabular}
\end{center}
We want to apply Theorem \ref{thm:maltsevtp1} to this Maltsev algebra.  Using suitable algebra morphisms, we can twist the Maltsev algebra $A$ into non-Hom-Lie, non-Maltsev Hom-Maltsev algebras or (Hom-)Lie algebras.

With a bit of computation, one can check that one class of algebra morphisms $\alpha_1 \colon A \to A$ is given by
\[
\begin{split}
\alpha_1(e_1) &= e_1 + a_3e_3 + a_4e_4,\\
\alpha_1(e_2) &= b_2e_2 + b_3e_3 + a_3b_2e_4,\\
\alpha_1(e_3) &= ce_3, \\
\alpha_1(e_4) &= b_2ce_4,
\end{split}
\]
where $a_3,a_4,b_2,b_3$, and $c$ are arbitrary scalars in $\bk$.  By Theorem \ref{thm:maltsevtp1} there is a Hom-Maltsev algebra
\[
A_{\alpha_1} = (A,[-,-]_{\alpha_1} = \alpha_1\circ[-,-],\alpha_1)
\]
with multiplication table:
\begin{center}
\begin{tabular}{c|cccc}
$[-,-]_{\alpha_1}$ & $e_1$ & $e_2$ & $e_3$ & $e_4$ \\ \hline
$e_1$ & $0$ & $-\alpha_1(e_2)$ & $-ce_3$ & $b_2ce_4$ \\
$e_2$ & $\alpha_1(e_2)$ & $0$ & $2b_2ce_4$ & $0$ \\
$e_3$ & $ce_3$ & $-2b_2ce_4$ & $0$ & $0$ \\
$e_4$ & $-b_2ce_4$ & $0$ & $0$ & $0$ \\
\end{tabular}
\end{center}
Note that $A_{\alpha_1}$ is in general not a Hom-Lie algebra, i.e., $J_{A_{\alpha_1}} \not= 0$.  Indeed, we have
\[
J_A(e_1,e_2,e_3) = -6e_4.
\]
Combining this with \eqref{jaalphaj} we obtain
\[
\begin{split}
J_{A_{\alpha_1}}(e_1,e_2,e_3)
&= \alpha_1^2(J_A(e_1,e_2,e_3))\\
&= \alpha_1^2(-6e_4)\\
&= -6(b_2c)^2e_4,
\end{split}
\]
which is not equal to $0$ in general.

Also, $(A,[-,-]_{\alpha_1})$ is not a Maltsev algebra in general.  Indeed, let $J'$ denote the usual Jacobian of $(A,[-,-]_{\alpha_1})$, i.e.,
\begin{equation}
\label{j'}
J'(x,y,z) = [[x,y]_\alpha,z]_\alpha + [[z,x]_\alpha,y]_\alpha + [[y,z]_\alpha,x]_\alpha,
\end{equation}
where $\alpha = \alpha_1$.  To see that the Maltsev identity \eqref{maltsev} does not hold in $(A,[-,-]_{\alpha_1})$, it suffices to show that
\begin{equation}
\label{4dnotmaltsev}
J'(e_1,e_2,[e_1,e_3]_\alpha) \not= [J'(e_1,e_2,e_3),e_1]_\alpha.
\end{equation}
Indeed, we have
\[
\begin{split}
J'(e_1,e_2,e_3) &= [[e_1,e_2]_\alpha,e_3]_\alpha + [[e_3,e_1]_\alpha,e_2]_\alpha + [[e_2,e_3]_\alpha,e_1]_\alpha\\
&= -2b_2c(b_2 + c + b_2c)e_4.
\end{split}
\]
This implies that, on the one hand,
\[
\begin{split}
J'(e_1,e_2,[e_1,e_3]_\alpha) &= -cJ'(e_1,e_2,e_3)\\
&= 2b_2c^2(b_2 + c + b_2c)e_4.
\end{split}
\]
On the other hand, we have
\[
\begin{split}
[J'(e_1,e_2,e_3),e_1]_\alpha &= -2b_2c(b_2 + c + b_2c)[e_4,e_1]_\alpha\\
&= 2(b_2c)^2(b_2 + c + b_2c)e_4.
\end{split}
\]
This proves that \eqref{4dnotmaltsev} holds whenever $b_2 \not= 1$ and $b_2c(b_2 + c + b_2c) \not= 0$.  So $(A,[-,-]_{\alpha_1})$ is not a Maltsev algebra in general.

Using Theorem \ref{thm:maltsevtp1} it is also possible to twist the Maltsev algebra $A$ into a (Hom-)Lie algebra.  For example, consider the following class of algebra morphisms $\alpha_2 \colon A \to A$:
\[
\begin{split}
\alpha_2(e_1) &= -e_1 + a_2e_2 + a_3e_3 + a_4e_4,\\
\alpha_2(e_2) &= be_4,\\
\alpha_2(e_3) &= 0 = \alpha_2(e_4),
\end{split}
\]
where $a_2,a_3,a_4$, and $b$ are arbitrary scalars in $\bk$.  By Theorem \ref{thm:maltsevtp1} there is a Hom-Maltsev algebra
\[
A_{\alpha_2} = (A,[-,-]_{\alpha_2} = \alpha_2\circ[-,-],\alpha_2)
\]
with multiplication table:
\begin{center}
\begin{tabular}{c|cccc}
$[-,-]_{\alpha_2}$ & $e_1$ & $e_2$ & $e_3$ & $e_4$ \\ \hline
$e_1$ & $0$ & $-be_4$ & $0$ & $0$ \\
$e_2$ & $be_4$ & $0$ & $0$ & $0$ \\
$e_3$ & $0$ & $0$ & $0$ & $0$ \\
$e_4$ & $0$ & $0$ & $0$ & $0$ \\
\end{tabular}
\end{center}
From this multiplication table, it is easy to check that $J_{A_{\alpha_2}} = 0$, i.e., $A_{\alpha_2}$ is actually a Hom-Lie algebra.  Likewise, one can check that $(A,[-,-]_{\alpha_2})$ is a Lie algebra.
\qed
\end{example}

\begin{example}
\label{ex:5dmatlsev}
There is a five-dimensional non-Lie Maltsev algebra $(A,[-,-])$ \cite{sagle} (Example 3.4) with basis $\{e_1,e_2,e_3,e_4,e_5\}$ and multiplication table:
\begin{center}
\begin{tabular}{c|ccccc}
$[-,-]$ & $e_1$ & $e_2$ & $e_3$ & $e_4$ & $e_5$ \\ \hline
$e_1$ & $0$ & $0$ & $0$ & $e_2$ & $0$ \\
$e_2$ & $0$ & $0$ & $0$ & $0$ & $e_3$ \\
$e_3$ & $0$ & $0$ & $0$ & $0$ & $0$ \\
$e_4$ & $-e_2$ & $0$ & $0$ & $0$ & $0$ \\
$e_5$ & $0$ & $-e_3$ & $0$ & $0$ & $0$
\end{tabular}
\end{center}
Let us classify all the algebra morphisms $\alpha \colon A \to A$. From the multiplication table of $A$, it follows that $\alpha$ is determined by its values at $e_1$, $e_4$, and $e_5$.  With a bit of computation, one can show that $\alpha \colon A \to A$ is an algebra morphism if and only if it has the form
\[
\begin{split}
\alpha(e_1) &= a_1e_1 + a_2e_2 + a_3e_3 + a_4e_4 + a_5e_5,\\
\alpha(e_2) &= (a_1b_4 - a_4b_1)e_2 + (a_2b_5 - a_5b_2)e_3,\\
\alpha(e_3) &= (a_1b_4 - a_4b_1)c_5e_3,\\
\alpha(e_4) &= b_1e_1 + b_2e_2 + b_3e_3 + b_4e_4 + b_5e_5,\\
\alpha(e_5) &= c_1e_1 + c_2e_2 + c_3e_3 + c_4e_4 + c_5e_5
\end{split}
\]
with $a_i, b_j, c_k \in \bk$, such that
\[
\begin{split}
a_5(a_4b_1 - a_1b_4) &= 0 = b_5(a_4b_1 - a_1b_4),\\
a_1c_4 = a_4c_1,\quad a_2c_5 &= a_5c_2,\quad b_1c_4 = b_4c_1,\quad b_2c_5 = b_5c_2.
\end{split}
\]
For each such algebra morphism $\alpha \colon A \to A$, by Theorem \ref{thm:maltsevtp1} there is a Hom-Maltsev algebra
\[
A_\alpha = (A,[-,-]_\alpha = \alpha\circ[-,-],\alpha)
\]
whose multiplication table is:
\begin{center}
\begin{tabular}{c|ccccc}
$[-,-]_\alpha$ & $e_1$ & $e_2$ & $e_3$ & $e_4$ & $e_5$ \\ \hline
$e_1$ & $0$ & $0$ & $0$ & $\alpha(e_2)$ & $0$ \\
$e_2$ & $0$ & $0$ & $0$ & $0$ & $\alpha(e_3)$ \\
$e_3$ & $0$ & $0$ & $0$ & $0$ & $0$ \\
$e_4$ & $-\alpha(e_2)$ & $0$ & $0$ & $0$ & $0$ \\
$e_5$ & $0$ & $-\alpha(e_3)$ & $0$ & $0$ & $0$
\end{tabular}
\end{center}
The Hom-Maltsev algebra $A_\alpha$ is in general not Hom-Lie, since
\[
\begin{split}
J_{A_\alpha}(e_1,e_4,e_5) &= \alpha^2(J_A(e_1,e_4,e_5)) \quad\text{(by \eqref{jaalphaj})}\\
&= \alpha^2(e_3)\\
&= (a_1b_4 - a_4b_1)^2c_5^2e_3,
\end{split}
\]
which is not equal to $0$ whenever $(a_1b_4 - a_4b_1)c_5 \not= 0$.

In strong contrast with Example \ref{ex:4dmaltsev}, we claim that $(A,[-,-]_\alpha)$ is \emph{always} a Maltsev algebra, regardless of what algebra morphism $\alpha \colon A \to A$ we choose.  Since $[-,-]_\alpha$ is anti-symmetric, we only need to see that the Maltsev identity \eqref{maltsev} holds.  Indeed, the images of $[-,-]$ and $[-,-]_\alpha$  are both contained in $span\{e_2,e_3\}$, and $\alpha(e_3)$ lies in $span\{e_3\}$.  It follows that
\[
[[x,y]_\alpha,z]_\alpha \subseteq span\{e_3\},
\]
which implies that
\begin{equation}
\label{xyzw}
[[[x,y]_\alpha,z]_\alpha,w]_\alpha = 0
\end{equation}
for all $w,x,y,z \in A$.  Thus, if $J'$ denotes the usual Jacobian of $(A,[-,-]_\alpha)$ as in \eqref{j'}, then
\begin{equation}
\label{j'1}
[J'(x,y,z),w]_\alpha = 0,
\end{equation}
which is in particular true when $w = x$.  On the other hand, we have
\[
\begin{split}
J'(x,y,[x,z]_\alpha) &= [[x,y]_\alpha,[x,z]_\alpha]_\alpha + [[[x,z]_\alpha,x]_\alpha,y]_\alpha + [[y,[x,z]_\alpha]_\alpha,x]_\alpha\\
&= \alpha^2([[x,y],[x,z]]) \quad\text{(by \eqref{xyzw} and multiplicativity of $\alpha$)}.
\end{split}
\]
Since both $[x,y]$ and $[x,z]$ lie in $span\{e_2,e_3\}$, it follows from the multiplication table of $A$ that $[[x,y],[x,z]] = 0$.  Thus, we have
\begin{equation}
\label{j'2}
J'(x,y,[x,z]_\alpha) = 0
\end{equation}
for all $x,y,z \in A$.  It follows from \eqref{j'1} and \eqref{j'2} that the Maltsev identity \eqref{maltsev} holds in $(A,[-,-]_\alpha)$.

Moreover, certain choices of algebra morphisms $\beta \colon A \to A$ make $A_\beta$ into a Hom-Lie algebra and $(A,[-,-]_\beta)$ into a Lie algebra.  For example, consider the algebra morphism $\beta$ on $A$ given by
\[
\beta(e_1) = ae_1,\quad \beta(e_2) = abe_2,\quad \beta(e_4) = be_4,\quad\beta(e_3) = 0 = \beta(e_5),
\]
where $a, b\in \bk$ are arbitrary scalars.  The only non-zero brackets in $A_\beta$ involving the basis elements are
\[
[e_1,e_4]_\beta = abe_2 = -[e_4,e_1]_\beta.
\]
This implies that $J_{A_\beta}$ and $J'$ (the usual Jacobian in $(A,[-,-]_\beta)$) are both equal to $0$, so $A_\beta$ is a Hom-Lie algebra and $(A,[-,-]_\beta)$ is a Lie algebra.
\qed
\end{example}

\section{Hom-alternative algebras are Hom-Maltsev-admissible}
\label{sec:homalternative}

The main purpose of this section is to show that every Hom-alternative algebra \cite{mak} gives rise to a Hom-Maltsev algebra via the commutator bracket (Theorem \ref{thm:homaltmaltsev}).  This means that Hom-alternative algebras are all Hom-Maltsev-admissible algebras, generalizing the well-known fact that alternative algebras are Maltsev-admissible.  More properties of Hom-alternative algebras are considered in sections \ref{sec:jordan} and  \ref{sec:moufang}.  At the end of this section, we consider an eight-dimensional, non-Hom-Lie Hom-Maltsev algebra arising from the octonions (Example \ref{ex:oct}).

Let us begin with some relevant definitions.

\begin{definition}
\label{def:alternating}
Let $V$ be a $\bk$-module and $f \colon V^{\otimes n} \to V$ be an $n$-linear map for some $n \geq 2$.  We say that $f$ is \textbf{alternating} if
\[
f = \epsilon(\pi) f \circ \pi
\]
for each permutation $\pi$ on $n$ letters, where $\epsilon(\pi) \in \{\pm 1\}$ is the signature of $\pi$.
\end{definition}

Since we are working over a field $\bk$ of characteristic $0$, the following characterizations of alternating maps are well-known facts in basic linear algebra and group theory.  We, therefore, omit the proof.  We will use the following Lemma without further comment.

\begin{lemma}
\label{lem:alternating}
Let $V$ be a $\bk$-module and $f \colon V^{\otimes n} \to V$ be an $n$-linear map.  Then the following statements are equivalent:
\begin{enumerate}
\item
$f$ is alternating.
\item
$f(x_1,\ldots,x_n) = 0$ whenever $x_i = x_j$ for some $i \not= j$.
\item
$f = -f\circ\iota$ for each transposition $\iota$ on $n$ letters.
\item
$f = (-1)^{n-1}f\circ\xi$ and $f = -f\circ\eta$, where $\xi$ is the cyclic permutation $(12\cdots n)$ and $\eta$ is the adjacent transposition $(n-1,n)$.
\end{enumerate}
\end{lemma}

\begin{definition}
\label{def:homalt}
Let $(A,\mu,\alpha)$ be a Hom-algebra (Definition \ref{def:homalgebra}).  Then $A$ is called a:
\begin{enumerate}
\item
\textbf{Hom-associative algebra} \cite{ms} if $as_A = 0$, where $as_A$ is the Hom-associator \eqref{homassociator};
\item
\textbf{Hom-alternative algebra} \cite{mak} if $as_A$ is alternating;
\item
\textbf{Hom-flexible algebra} \cite{ms} if $as_A(x,y,x) = 0$ for all $x,y \in A$.
\end{enumerate}
\end{definition}

It follows from the above definitions that a Hom-associative algebra is also a Hom-alternative algebra and that a Hom-alternative algebra is also a Hom-flexible algebra.  Also, when $\alpha = Id$ in Definition \ref{def:homalt}, we recover the usual notions of associative, alternative, and flexible algebras, respectively.

\begin{remark}
In \cite{mak} a Hom-alternative algebra was actually defined as a Hom-algebra $(A,\mu,\alpha)$ that satisfies both
\begin{equation}
\label{lefthomalt}
as_A(x,x,y) = 0 \quad\text{(left Hom-alternativity)}
\end{equation}
and
\begin{equation}
\label{righthomalt}
as_A(x,y,y) = 0 \quad\text{(right Hom-alternativity)}
\end{equation}
for all $x,y \in A$.  It is shown in \cite{mak} that this definition is equivalent to the one in Definition \ref{def:homalt}.  Indeed, if $as_A$ is alternating, then it is clearly also left and right Hom-alternative.  Conversely, left (right) Hom-alternativity is equivalent to $as_A$ being alternating in the first (last) two variables.  Since the transpositions $(1~2)$ and $(2~3)$ generate the symmetric group on three letters, one infers that left and right Hom-alternativity together imply that $as_A$ is alternating.
\end{remark}

To state the main result of this section, we need the following definition.  Recall that $\tau \colon V \otimes W \cong W \otimes V$ denotes the twist isomorphism, $\tau(v \otimes w) = w \otimes v$.

\begin{definition}
\label{def:commutatoralgebra}
Let $(A,\mu,\alpha)$ be a Hom-algebra.  Define its \textbf{commutator Hom-algebra} as the Hom-algebra
\[
A^- = (A,[-,-] = \mu \circ (Id - \tau),\alpha).
\]
The multiplication $[-,-] = \mu \circ (Id - \tau)$ is called the \textbf{commutator bracket} of $\mu$.  We call a Hom-algebra $A$ \textbf{Hom-Maltsev-admissible} (resp. \textbf{Hom-Lie-admissible} \cite{ms}) if $A^-$ is a Hom-Maltsev (resp. Hom-Lie) algebra (Definition \ref{def:hommaltsev}).
\end{definition}

\begin{example}
\label{ex:hmadmissible}
Since Hom-Lie algebras are all Hom-Maltsev algebras (Example \ref{ex:homlie}), every Hom-Lie-admissible algebra is also Hom-Maltsev-admissible.  In particular, since every $G$-Hom-associative algebra (e.g., Hom-associative, Hom-Lie, or Hom-pre-Lie algebra) is Hom-Lie-admissible \cite{ms} (Proposition 2.7), it is also Hom-Maltsev-admissible.  In section \ref{sec:hommaltsevad}, we will give examples of Hom-Maltsev-admissible algebras that are not Hom-Lie-admissible.\qed
\end{example}

\begin{example}
\label{ex:malad}
A \textbf{Maltsev-admissible algebra} is defined as an algebra $(A,\mu)$ for which the commutator algebra $A^- = (A,[-,-]=\mu\circ(Id - \tau))$ is a Maltsev algebra, i.e., $A^-$ satisfies the Maltsev identity \eqref{maltsev} (or equivalently \eqref{maltsevidentity}).  Identifying algebras as Hom-algebras with identity twisting maps (Remark \ref{rk:classicalalgebra}), a Maltsev-admissible algebra is equivalent to a Hom-Maltsev-admissible algebra with $\alpha = Id$.\qed
\end{example}

It is proved in \cite{ms} that, given a Hom-associative algebra $A$, its commutator Hom-algebra $A^-$ is a Hom-Lie algebra.  Also, the commutator algebra of any alternative algebra is a Maltsev algebra.  The following main result of this section generalizes both of these facts.  It gives us a large class of Hom-Maltsev-admissible algebras that are in general not Hom-Lie-admissible.

\begin{theorem}
\label{thm:homaltmaltsev}
Every Hom-alternative algebra is Hom-Maltsev-admissible.
\end{theorem}

In particular, Hom-alternative algebras are all Hom-flexible, Hom-Maltsev-admissible algebras.  Examples of Hom-flexible, Hom-Maltsev-admissible algebras that are not Hom-alternative are considered in section \ref{sec:hommaltsevad}.

The proof of Theorem \ref{thm:homaltmaltsev} depends on the Hom-type analogues of some identities in alternative algebras, most of which are from \cite{bk}.  We will first establish some identities about the Hom-associator and the Hom-Jacobian.  Then we will go back to the proof of Theorem \ref{thm:homaltmaltsev}.  In what follows, we often write $\mu(a,b)$ as $ab$ and omit the subscript in the Hom-associator $as_A$ \eqref{homassociator} when there is no danger of confusion.

The following result is a sort of cocycle condition for the Hom-associator of a Hom-alternative algebra.

\begin{lemma}
\label{lem1:homalt}
Let $(A,\mu,\alpha)$ be a Hom-alternative algebra.  Then the identity
\begin{equation}
\label{ass3cocycle}
\begin{split}
as(wx,\alpha(y),\alpha(z)) &- as(xy,\alpha(z),\alpha(w)) + as(yz,\alpha(w),\alpha(x))\\
&= \alpha^2(w)as(x,y,z) + as(w,x,y)\alpha^2(z)
\end{split}
\end{equation}
holds for all $w,x,y,z \in A$.
\end{lemma}

\begin{proof}
First we claim that for \emph{any} Hom-algebra $(A,\mu,\alpha)$ (Hom-alternative or otherwise), we have
\begin{equation}
\label{homasscocycle}
\begin{split}
as(wx,\alpha(y),\alpha(z)) &- as(\alpha(w),xy,\alpha(z)) + as(\alpha(w),\alpha(x),yz)\\
&= \alpha^2(w)as(x,y,z) + as(w,x,y)\alpha^2(z)
\end{split}
\end{equation}
for all $w,x,y,z \in A$.  Indeed, starting from the left-hand side of \eqref{homasscocycle}, we have:
\[
\begin{split}
as(wx,\alpha(y),\alpha(z)) &- as(\alpha(w),xy,\alpha(z)) + as(\alpha(w),\alpha(x),yz)\\
&= ((wx)\alpha(y))\alpha^2(z) - \alpha(wx)(\alpha(y)\alpha(z)) - (\alpha(w)(xy))\alpha^2(z)\\
&\relphantom{} + \alpha^2(w)((xy)\alpha(z)) + (\alpha(w)\alpha(x))\alpha(yz) - \alpha^2(w)(\alpha(x)(yz))\\
&= \left\{(wx)\alpha(y) - \alpha(w)(xy)\right\}\alpha^2(z) + \alpha^2(w)\left\{(xy)\alpha(z) - \alpha(x)(yz)\right\}\\
&\relphantom{} - \alpha(wx)\alpha(yz) + \alpha(wx)\alpha(yz)\\
&= \alpha^2(w)as(x,y,z) + as(w,x,y)\alpha^2(z).
\end{split}
\]
In the second equality above, we used the multiplicativity of $\alpha$ twice. We have established \eqref{homasscocycle}.  Now for a Hom-alternative algebra $A$, its Hom-associator $as$ is alternating, so \eqref{homasscocycle} implies \eqref{ass3cocycle}.
\end{proof}

\begin{remark}
Note that when $\alpha = Id$, the condition \eqref{homasscocycle} says that the (Hom-)associator $as \in \Hom(A^{\otimes 3},A)$ is a Hochschild $3$-cocycle.
\end{remark}

In a Hom-alternative algebra, the Hom-associator is an alternating map on three variables.  We now build a map on four variables using the Hom-associator that, as we will prove shortly, is alternating in a Hom-alternative algebra.

\begin{definition}
\label{def:bk}
Let $(A,\mu,\alpha)$ be a Hom-algebra.  Define the \textbf{Hom-Bruck-Kleinfeld function} $f \colon A^{\otimes 4} \to A$ as the multi-linear map
\begin{equation}
\label{f}
f(w,x,y,z) = as(wx,\alpha(y),\alpha(z)) - as(x,y,z)\alpha^2(w) - \alpha^2(x)as(w,y,z)
\end{equation}
for $w,x,y,z \in A$.  Define another multi-linear map $F \colon A^{\otimes 4} \to A$ as
\begin{equation}
\label{F}
F = [-,-] \circ \left(\alpha^2 \otimes as\right) \circ (Id - \xi + \xi^2 - \xi^3),
\end{equation}
where $[-,-] = \mu \circ (Id - \tau)$ is the commutator bracket of $\mu$ and $\xi$ is the cyclic permutation
\[
\xi(w \otimes x \otimes y \otimes z) = z \otimes w \otimes x \otimes y.
\]
\end{definition}

In terms of elements, the map $F$ is given by
\begin{equation}
\label{F'}
\begin{split}
F(w,x,y,z) &= [\alpha^2(w),as(x,y,z)] - [\alpha^2(z),as(w,x,y)]\\
&\relphantom{} + [\alpha^2(y),as(z,w,x)] - [\alpha^2(x),as(y,z,w)].
\end{split}
\end{equation}
The Hom-Bruck-Kleinfeld function $f$ is the Hom-type analogue of a map studied by Bruck and Kleinfeld (\cite{bk} (2.7)).  It is closely related to the map $F$, as we now show.

\begin{lemma}
\label{lem2:homalt}
In a Hom-alternative algebra $(A,\mu,\alpha)$, we have
\[
F = f \circ (Id - \rho + \rho^2),
\]
where $\rho = \xi^3$ is the cyclic permutation $\rho(w \otimes x \otimes y \otimes z) = x \otimes y \otimes z \otimes w$.
\end{lemma}

\begin{proof}
By Lemma \ref{lem1:homalt} and \eqref{f}, we have
\[
\begin{split}
\alpha^2(w) as(x,y,z) &+ as(w,x,y)\alpha^2(z)\\
&= as(wx,\alpha(y),\alpha(z)) - as(xy,\alpha(z),\alpha(w)) + as(yz,\alpha(w),\alpha(x))\\
&= f(w,x,y,z) + as(x,y,z)\alpha^2(w) + \alpha^2(x)as(w,y,z)\\
&\relphantom{} - f(x,y,z,w) - as(y,z,w)\alpha^2(x) - \alpha^2(y)as(x,z,w)\\
&\relphantom{} + f(y,z,w,x) + as(z,w,x)\alpha^2(y) + \alpha^2(z)as(y,w,x).
\end{split}
\]
Since the Hom-associator $as$ is alternating, we have $as(y,w,x) = as(w,x,y)$, $as(x,z,w) = as(z,w,x)$, and $as(w,y,z) = as(y,z,w)$.  Therefore, rearranging terms in the above equality, we obtain $F = f \circ (Id - \rho + \rho^2)$ in the explicit form \eqref{F'}.
\end{proof}

The following result is the Hom-type analogue of part of \cite{bk} (Lemma 2.1).

\begin{proposition}
\label{prop:f}
Let $(A,\mu,\alpha)$ be a Hom-alternative algebra.  Then the Hom-Bruck-Kleinfeld function $f$ is alternating.
\end{proposition}

\begin{proof}
First observe that $-F = F \circ \xi$, which follows immediately from the definition \eqref{F} of $F$.  This implies $-F = F \circ \rho$, where $\rho = \xi^3$.  Note that $\rho^3 = \xi$.  Thus, we have:
\[
\begin{split}
0 &= F \circ (Id + \rho)\\
&= f \circ (Id - \rho + \rho^2) \circ (Id + \rho) \quad\text{(by Lemma \ref{lem2:homalt})}\\
&= f \circ (Id + \rho^3)\\
&= f \circ (Id + \xi).
\end{split}
\]
Equivalently, we have
\begin{equation}
\label{fxi}
f = -f\circ\xi,
\end{equation}
so $f$ changes sign under the cyclic permutation $\xi$.  From the definition \eqref{f} of $f$ and the fact that the Hom-associator $as$ is alternating in a Hom-alternative algebra, we infer also that
\begin{equation}
\label{feta}
f = -f\circ\eta,
\end{equation}
where $\eta$ is the adjacent transposition $\eta(w \otimes x \otimes y \otimes z) = w \otimes x \otimes z \otimes y$.  So $f$ changes sign under the transposition $\eta$.  Since the cyclic permutation $\xi$ and the adjacent transposition $\eta$ generate the symmetric group on four letters, we infer from \eqref{fxi} and \eqref{feta} that $f$ is alternating.
\end{proof}

The following identities are consequences of Proposition \ref{prop:f} and are the Hom-type analogues of part of \cite{bk} (Lemma 2.2).  In what follows, we write $\mu(x,x)$ as $x^2$.

\begin{corollary}
\label{cor1:homalt}
Let $(A,\mu,\alpha)$ be a Hom-alternative algebra.  Then:
\begin{equation}
\label{homalt1}
as(x^2,\alpha(y),\alpha(z)) = \alpha^2(x)as(x,y,z) + as(x,y,z)\alpha^2(x),
\end{equation}
\begin{equation}
\label{homalt2}
as(\alpha(x),xy,\alpha(z)) = as(x,y,z)\alpha^2(x) = as(\alpha(x),\alpha(y),xz),
\end{equation}
and
\begin{equation}
\label{homalt3}
as(\alpha(x),yx,\alpha(z)) = \alpha^2(x)as(x,y,z) = as(\alpha(x),\alpha(y),zx)
\end{equation}
for all $x,y,z \in A$.
\end{corollary}

\begin{proof}
To obtain \eqref{homalt1}, set $w = x$ in the definition \eqref{f} of $f$ and use the fact that $f$ is alternating (Proposition \ref{prop:f}).

For \eqref{homalt2} we compute as follows:
\[
\begin{split}
as(\alpha(x),xy,\alpha(z)) &= as(xy,\alpha(z),\alpha(x)) \quad\text{(by alternativity of $as$)}\\
&= f(x,y,z,x) + as(y,z,x)\alpha^2(x) + \alpha^2(y)as(x,z,x)\quad\text{(by \eqref{f})}\\
&= as(x,y,z)\alpha^2(x) \quad\text{(by alternativity of $f$ and $as$)}.
\end{split}
\]
This proves half of \eqref{homalt2}.  For the other half of \eqref{homalt2}, we compute similarly as follows:
\[
\begin{split}
as(\alpha(x),\alpha(y),xz) &= as(xz,\alpha(x),\alpha(y))\quad\text{(by alternativity of $as$)}\\
&= f(x,z,x,y) + as(z,x,y)\alpha^2(x) + \alpha^2(z)as(x,x,y)\quad\text{(by \eqref{f})}\\
&= as(x,y,z)\alpha^2(x)\quad\text{(by alternativity of $f$ and $as$)}.
\end{split}
\]
This finishes the proof of \eqref{homalt2}.

For \eqref{homalt3} we compute as follows:
\[
\begin{split}
as(\alpha(x),yx,\alpha(z)) &= as(yx,\alpha(z),\alpha(x))\quad\text{(by alternativity of $as$)}\\
&= f(y,x,z,x) + as(x,z,x)\alpha^2(y) + \alpha^2(x)as(y,z,x)\quad\text{(by \eqref{f})}\\
&= \alpha^2(x)as(x,y,z)\quad\text{(by alternativity of $f$ and $as$)}.
\end{split}
\]
This proves half of \eqref{homalt3}.  For the other half of \eqref{homalt3}, we compute similarly as follows:
\[
\begin{split}
as(\alpha(x),\alpha(y),zx) &= as(zx,\alpha(x),\alpha(y))\quad\text{(by alternativity of $as$)}\\
&= f(z,x,x,y) + as(x,x,y)\alpha^2(z) + \alpha^2(x)as(z,x,y)\quad\text{(by \eqref{f})}\\
&= \alpha^2(x)as(x,y,z)\quad\text{(by alternativity of $f$ and $as$)}.
\end{split}
\]
This finishes the proof.
\end{proof}

The following result says that every Hom-alternative algebra satisfies a variation of the Hom-Maltsev identity \eqref{hommaltsevid} in which the Hom-Jacobian is replaced by the Hom-associator.

\begin{corollary}
\label{cor:asxyxz}
Let $(A,\mu,\alpha)$ be a Hom-alternative algebra.  Then
\[
as(\alpha(x),\alpha(y),[x,z]) = [as(x,y,z),\alpha^2(x)]
\]
for all $x,y,z \in A$, where $[-,-] = \mu\circ(Id-\tau)$ is the commutator bracket.
\end{corollary}

\begin{proof}
Indeed, we have
\[
\begin{split}
as(\alpha(x),\alpha(y),[x,z])
&= as(\alpha(x),\alpha(y), xz) - as(\alpha(x),\alpha(y),zx)\\
&= as(x,y,z)\alpha^2(x) - \alpha^2(x)as(x,y,z) \quad\text{(by \eqref{homalt2} and \eqref{homalt3})}\\
&= [as(x,y,z),\alpha^2(x)],
\end{split}
\]
as desired.
\end{proof}

Next we consider the relationship between the Hom-associator \eqref{homassociator} in a Hom-algebra $A$ and the Hom-Jacobian \eqref{homjacobian} in its commutator Hom-algebra $A^-$ (Definition \ref{def:commutatoralgebra}).

\begin{lemma}
\label{lem:hahj}
Let $(A,\mu,\alpha)$ be any Hom-algebra.  Then
\[
J_{A^-} = as_A \circ (Id + \sigma + \sigma^2) \circ (Id - \delta),
\]
where $\sigma(x \otimes y \otimes z) = z \otimes x \otimes y$ and $\delta(x \otimes y \otimes z) = x \otimes z \otimes y$.
\end{lemma}

\begin{proof}
For $x,y,z \in A$, we have:
\[
\begin{split}
J_{A^-}(x,y,z) &= [[x,y],\alpha(z)] + [[z,x],\alpha(y)] + [[y,z],\alpha(x)]\\
&= (xy)\alpha(z) - (yx)\alpha(z) - \alpha(z)(xy) + \alpha(z)(yx)\\
&\relphantom{} + (zx)\alpha(y) - (xz)\alpha(y) - \alpha(y)(zx) + \alpha(y)(xz)\\
&\relphantom{} + (yz)\alpha(x) - (zy)\alpha(x) - \alpha(x)(yz) + \alpha(x)(zy)\\
&= as_A(x,y,z) + as_A(z,x,y) + as_A(y,z,x)\\
&\relphantom{} - as_A(x,z,y) - as_A(y,x,z) - as_A(z,y,x)\\
&= as_A \circ (Id + \sigma + \sigma^2) \circ (Id - \delta)(x,y,z).
\end{split}
\]
This proves the Lemma.
\end{proof}

\begin{proposition}
\label{prop2:homalt}
Let $(A,\mu,\alpha)$ be a Hom-alternative algebra.  Then we have $J_{A^-} = 6as_A$.
\end{proposition}

\begin{proof}
Since the Hom-associator $as_A$ is alternating, with the notations in  Lemma \ref{lem:hahj} we have
\[
as_A \circ \sigma = as_A = -as_A \circ \delta.
\]
The result now follows from Lemma \ref{lem:hahj}.
\end{proof}

We are now ready to prove Theorem \ref{thm:homaltmaltsev}.

\begin{proof}[Proof of Theorem \ref{thm:homaltmaltsev}]
Let $(A,\mu,\alpha)$ be a Hom-alternative algebra and $A^- = (A,[-,-],\alpha)$ be its commutator Hom-algebra.  The commutator bracket $[-,-]=\mu\circ(Id - \tau)$ is anti-symmetric.  Thus, it remains to show that the Hom-Maltsev identity \eqref{hommaltsevid} holds in $A^-$, i.e.,
\[
J_{A^-}(\alpha(x),\alpha(y),[x,z]) = [J_{A^-}(x,y,z),\alpha^2(x)].
\]
To prove this, we compute as follows:
\[
\begin{split}
J_{A^-}(\alpha(x),\alpha(y),[x,z])
&= 6 as_A(\alpha(x),\alpha(y),[x,z]) \quad\text{(by Proposition \ref{prop2:homalt})}\\
&= [6as_A(x,y,z),\alpha^2(x)]\quad\text{(by Corollary \ref{cor:asxyxz})}\\
&= [J_{A^-}(x,y,z),\alpha^2(x)]\quad\text{(by Proposition \ref{prop2:homalt})}.
\end{split}
\]
We have shown that $A^-$ is a Hom-Maltsev algebra, so $A$ is Hom-Maltsev-admissible.
\end{proof}

\begin{remark}
By Theorem \ref{thm:homaltmaltsev} the map $A \mapsto A^-$ defines a functor from the category of Hom-alternative algebras to the category of Hom-Maltsev algebras.  Using the combinatorial objects of weighted trees and an argument similar to that in \cite{yau}, one can show that this functor has a left adjoint $M \mapsto U(M)$.  However, we do not know whether there is an analogue of the Poincar\'{e}-Birkhoff-Witt (PBW) Theorem, i.e., whether the canonical map $M \to U(M)$ is injective.  In fact, even in the non-Hom case of Maltsev algebras, it is not known whether there is a PBW Theorem with alternative algebras in place of associative algebras.  Probably the closest result to a PBW Theorem for Maltsev algebras is in \cite{ps}.
\end{remark}

\begin{example}
\label{ex:oct}
In this example, we describe a Hom-alternative algebra (hence Hom-Maltsev-admissible by Theorem \ref{thm:homaltmaltsev}) that is not Hom-Lie-admissible and not alternative.  Recall that the octonions is an eight-dimensional alternative (but not associative) algebra $\oct$ with basis $\{e_0,\ldots,e_7\}$ and the following multiplication table, where $\mu$ denotes the multiplication in $\oct$.
\begin{center}
\begin{tabular}{c|c|c|c|c|c|c|c|c}
$\mu$ & $e_0$ & $e_1$ & $e_2$ & $e_3$ & $e_4$ & $e_5$ & $e_6$ & $e_7$ \\\hline
$e_0$ & $e_0$ & $e_1$ & $e_2$ & $e_3$ & $e_4$ & $e_5$ & $e_6$ & $e_7$ \\\hline
$e_1$ & $e_1$ & $-e_0$ & $e_4$ & $e_7$ & $-e_2$ & $e_6$ & $-e_5$ & $-e_3$ \\\hline
$e_2$ & $e_2$ & $-e_4$ & $-e_0$ & $e_5$ & $e_1$ & $-e_3$ & $e_7$ & $-e_6$ \\\hline
$e_3$ & $e_3$ & $-e_7$ & $-e_5$ & $-e_0$ & $e_6$ & $e_2$ & $-e_4$ & $e_1$ \\\hline
$e_4$ & $e_4$ & $e_2$ & $-e_1$ & $-e_6$ & $-e_0$ & $e_7$ & $e_3$ & $-e_5$ \\\hline
$e_5$ & $e_5$ & $-e_6$ & $e_3$ & $-e_2$ & $-e_7$ & $-e_0$ & $e_1$ & $e_4$ \\\hline
$e_6$ & $e_6$ & $e_5$ & $-e_7$ & $e_4$ & $-e_3$ & $-e_1$ & $-e_0$ & $e_2$ \\\hline
$e_7$ & $e_7$ & $e_3$ & $e_6$ & $-e_1$ & $e_5$ & $-e_4$ & $-e_2$ & $-e_0$ \\
\end{tabular}
\end{center}
The reader is referred to \cite{baez,gt,okubo,sv} for discussion about the roles of the octonions in exceptional Lie groups, projective geometry, physics, and other applications.

One can check that there is an algebra automorphism $\alpha \colon \oct \to \oct$ given by
\begin{equation}
\label{octaut}
\begin{split}
\alpha(e_0) = e_0,\quad \alpha(e_1) = e_5,\quad \alpha(e_2) = e_6,\quad \alpha(e_3) = e_7,\\
\alpha(e_4) = e_1,\quad \alpha(e_5) = e_2,\quad \alpha(e_6) = e_3,\quad \alpha(e_7) = e_4.
\end{split}
\end{equation}
There is a more conceptual description of this algebra automorphism on $\oct$.  Note that $e_1$ and $e_2$ anti-commute, and $e_3$ anti-commutes with $e_1$, $e_2$, and $e_1e_2 = e_4$. Such a triple $(e_1,e_2,e_3)$ is called a \textbf{basic triple} in \cite{baez}.  Another basic triple is $(e_5,e_6,e_7)$.  Then $\alpha$ is the unique automorphism on $\oct$ that sends the basic triple $(e_1,e_2,e_3)$ to the basic triple $(e_5,e_6,e_7)$.

Using \cite{mak} (Theorem 2.1), which is the analogue of Theorem \ref{thm:maltsevtp1} for Hom-alternative algebras, we obtain a Hom-alternative (hence Hom-Maltsev-admissible) algebra
\[
\octalpha = (\oct,\mualpha = \alpha\circ\mu,\alpha)
\]
with the following multiplication table.
\begin{center}
\begin{tabular}{c|c|c|c|c|c|c|c|c}
$\mualpha$ & $e_0$ & $e_1$ & $e_2$ & $e_3$ & $e_4$ & $e_5$ & $e_6$ & $e_7$ \\\hline
$e_0$ & $e_0$ & $e_5$ & $e_6$ & $e_7$ & $e_1$ & $e_2$ & $e_3$ & $e_4$\\\hline
$e_1$ & $e_5$ & $-e_0$ & $e_1$ & $e_4$ & $-e_6$ & $e_3$ & $-e_2$ & $-e_7$\\\hline
$e_2$ & $e_6$ & $-e_1$ & $-e_0$ & $e_2$ & $e_5$ & $-e_7$ & $e_4$ & $-e_3$\\\hline
$e_3$ & $e_7$ & $-e_4$ & $-e_2$ & $-e_0$ & $e_3$ & $e_6$ & $-e_1$ & $e_5$\\\hline
$e_4$ & $e_1$ & $e_6$ & $-e_5$ & $-e_3$ & $-e_0$ & $e_4$ & $e_7$ & $-e_2$\\\hline
$e_5$ & $e_2$ & $-e_3$ & $e_7$ & $-e_6$ & $-e_4$ & $-e_0$ & $e_5$ & $e_1$\\\hline
$e_6$ & $e_3$ & $e_2$ & $-e_4$ & $e_1$ & $-e_7$ & $-e_5$ & $-e_0$ & $e_6$\\\hline
$e_7$ & $e_4$ & $e_7$ & $e_3$ & $-e_5$ & $e_2$ & $-e_1$ & $-e_6$ & $-e_0$\\
\end{tabular}
\end{center}
Note that $\octalpha$ is not alternative because
\[
\mualpha(\mualpha(e_0,e_0),e_1) = e_5 \not= e_2 = \mualpha(e_0,\mualpha(e_0,e_1)).
\]
Since $\octalpha$ is Hom-alternative, by Theorem \ref{thm:homaltmaltsev} it is also Hom-Maltsev-admissible, i.e., its commutator Hom-algebra
\[
\octalphaminus = (\oct,[-,-]_\alpha,\alpha),
\]
where $[-,-]_\alpha = \mualpha \circ (Id - \tau) = \alpha\circ\mu\circ(Id - \tau)$, is a Hom-Maltsev algebra.

Observe that $\octalpha$ is not Hom-Lie-admissible, i.e., $\octalphaminus$ is not a Hom-Lie algebra.  Indeed, with $\octminus$ denoting the Maltsev algebra $(\oct,[-,-]=\mu\circ(Id - \tau))$, we have
\begin{equation}
\label{jo}
\begin{split}
J_{\octalphaminus} &= [-,-]_\alpha \circ ([-,-]_\alpha \otimes \alpha) \circ (Id + \sigma + \sigma^2)\\
&= \alpha^2 \circ [-,-] \circ ([-,-] \otimes Id) \circ (Id + \sigma + \sigma^2)\\
&= \alpha^2 \circ J_{\octminus}\\
&= 6\alpha^2\circ as_{\oct}
\end{split}
\end{equation}
by Proposition \ref{prop2:homalt}.  Here we are regarding $\oct$ as the Hom-alternative algebra $(\oct,\mu,Id)$ with identity twisting map.  Since $as_{\oct} \not= 0$ (because $\oct$ is not associative) and $\alpha$ is an automorphism, it follows that $J_{\octalphaminus} \not= 0$.  For example, we have
\[
as_{\oct}(e_5,e_6,e_7) = -2e_3 \not= 0
\]
and
\[
J_{\octalphaminus}(e_5,e_6,e_7) = -12\alpha^2(e_3) = -12e_4 \not= 0.
\]
Therefore, $\octalpha$ is a Hom-alternative (and hence Hom-Maltsev-admissible) algebra that is neither alternative nor Hom-Lie-admissible.

Also, $(\oct,[-,-]_\alpha)$ is not a Maltsev algebra.  Indeed, let $J'$ denote the usual Jacobian in $(\oct,[-,-]_\alpha)$ as in \eqref{j'}.  Then we have
\[
\begin{split}
[J'(e_5,e_6,e_7),e_5]_\alpha &= [[[e_5,e_6]_\alpha,e_7]_\alpha,e_5]_\alpha + [[[e_7,e_5]_\alpha,e_6]_\alpha,e_5]_\alpha + [[[e_6,e_7]_\alpha,e_5]_\alpha,e_5]_\alpha\\
&= 8(e_3 - e_7)
\end{split}
\]
and
\[
\begin{split}
J'(e_5,e_6,[e_5,e_7]_\alpha) &= 2J'(e_5,e_6,e_1)\\
&= -8(e_1+e_3+e_7).
\end{split}
\]
So $(\oct,[-,-]_\alpha)$ does not satisfy the Maltsev identity \eqref{maltsev}.

Finally, observe that as long as $\beta \colon \oct \to \oct$ is an algebra automorphism, \eqref{jo} implies that $\oct_\beta^-$ is a Hom-Maltsev algebra that is not Hom-Lie.  There are plenty of algebra automorphisms on $\oct$ other than the one in \eqref{octaut}.  In fact, the automorphism group of $\oct$ is the $14$-dimensional exceptional Lie group $G_2$ \cite{baez,cartan}.
\qed
\end{example}

\section{Hom-Maltsev-admissible algebras}
\label{sec:hommaltsevad}

In Theorem \ref{thm:homaltmaltsev} we showed that every Hom-alternative algebra is Hom-Maltsev-admissible, i.e., its commutator Hom-algebra is a Hom-Maltsev algebra.  Since Hom-alternative algebras are always Hom-flexible (Definition \ref{def:homalt}), we know that Hom-alternative algebras are Hom-flexible, Hom-Maltsev-admissible algebras.  The purpose of this section is to study the (strictly larger) class of Hom-flexible, Hom-Maltsev-admissible algebras.  We give several characterizations of Hom-flexible algebras that are Hom-Maltsev-admissible in terms of the cyclic Hom-associator (Proposition \ref{prop:hmad}).  Then we prove the analogues of the construction results, Theorems  \ref{thm:maltsevtp2} and \ref{thm:maltsevtp1}, for Hom-flexible and Hom-Maltsev-admissible algebras (Theorems \ref{thm:homf} and \ref{thm:hommalad}).  We then consider examples of Hom-flexible, Hom-Maltsev-admissible algebras that are neither Hom-alternative nor Hom-Lie-admissible.

To state our characterizations of Hom-flexible algebras that are Hom-Maltsev-admissible, we need the following definition.

\begin{definition}
\label{def:S}
Let $(A,\mu,\alpha)$ be a Hom-algebra.  Define the \textbf{cyclic Hom-associator} $S_A \colon A^{\otimes 3} \to A$ as the multi-linear map
\[
S_A = as_A \circ (Id + \sigma + \sigma^2),
\]
where $as_A$ is the Hom-associator \eqref{homassociator} and $\sigma(x \otimes y \otimes z) = z \otimes x \otimes y$.
\end{definition}

We will use the following preliminary observations about the relationship between the cyclic Hom-associator and the Hom-Jacobian \eqref{homjacobian} of the commutator Hom-algebra (Definition \ref{def:commutatoralgebra}).

\begin{lemma}
\label{lem:S}
Let $(A,\mu,\alpha)$ be a Hom-flexible algebra.  Then we have
\begin{equation}
\label{SJ2}
2S_A = J_{A^-},
\end{equation}
where $A^- = (A,[-,-],\alpha)$ is the commutator Hom-algebra.
\end{lemma}

\begin{proof}
Let $\cyclicsum$ denote the cyclic sum $(Id + \sigma + \sigma^2)$.  We have:
\begin{equation}
\label{sadelta}
\begin{split}
S_A(x,y,z) &= \cyclicsum as_A(x,y,z)\\
&= - \cyclicsum as_A(z,y,x) \quad\text{(by Hom-flexibility)}\\
&= - \cyclicsum as_A(x,z,y)\\
&= -S_A(x,z,y).
\end{split}
\end{equation}
Let $\delta$ denote the permutation $\delta(x \otimes y \otimes z) = x \otimes z \otimes y$.  Then \eqref{sadelta} is equivalent to
\[
S_A = -S_A\circ\delta.
\]
This implies that
\[
\begin{split}
2S_A &= S_A \circ (Id - \delta)\\
&= as_A \circ (Id + \sigma + \sigma^2) \circ (Id - \delta)\\
&= J_{A^-},
\end{split}
\]
where the last equality is by Lemma \ref{lem:hahj}.
\end{proof}

The following result gives characterizations of Hom-Maltsev-admissible algebras in terms of the cyclic Hom-associator, assuming Hom-flexibility.  The condition \eqref{sa} below is the Hom-type analogue of \cite{myung} (Lemma 1.2(ii)).

\begin{proposition}
\label{prop:hmad}
Let $(A,\mu,\alpha)$ be a Hom-flexible algebra and $A^- = (A,[-,-],\alpha)$ be its commutator Hom-algebra.  Then the following statements are equivalent:
\begin{enumerate}
\item
$A$ is Hom-Maltsev-admissible (Definition \ref{def:commutatoralgebra}).
\item
The equality
\begin{equation}
\label{jaminus}
J_{A^-}(\alpha(x),\alpha(y),[x,z]) = [J_{A^-}(x,y,z),\alpha^2(x)]
\end{equation}
holds for all $x,y,z \in A$.
\item
The equality
\begin{equation}
\label{sa}
S_A(\alpha(x),\alpha(y),[x,z]) = [S_A(x,y,z),\alpha^2(x)]
\end{equation}
holds for all $x,y,z \in A$.
\item
The equality
\begin{equation}
\label{sa'}
\begin{split}
S_A(\alpha(w),\alpha(y),[x,z]) &+ S_A(\alpha(x),\alpha(y),[w,z])\\
&= [S_A(w,y,z),\alpha^2(x)] + [S_A(x,y,z),\alpha^2(w)]
\end{split}
\end{equation}
holds for all $w,x,y,z \in A$.
\end{enumerate}
\end{proposition}

\begin{proof}
The equivalence of the first two statements is immediate, since the commutator bracket $[-,-] = \mu \circ (Id - \tau)$ is anti-symmetric and \eqref{jaminus} is the Hom-Maltsev identity \eqref{hommaltsevid} for $A^-$.  The equivalence of \eqref{jaminus} and \eqref{sa} follows from \eqref{SJ2}, which uses the Hom-flexibility of $A$.  Finally, that \eqref{sa} is equivalent to \eqref{sa'} follows from linearization.  In other words, starting from \eqref{sa}, one replaces $x$ by $w+x$ to obtain \eqref{sa'}.  Conversely, starting from \eqref{sa'}, one sets $w=x$ to obtain \eqref{sa}.
\end{proof}

The following construction results for Hom-flexible and Hom-Maltsev-admissible algebras are the analogues of Theorems  \ref{thm:maltsevtp2} and \ref{thm:maltsevtp1}.

\begin{theorem}
\label{thm:homf}
\begin{enumerate}
\item
Let $(A,\mu)$ be a flexible algebra (i.e., $(xy)x = x(yx)$ for all $x,y \in A$) and $\alpha \colon A \to A$ be an algebra morphism.  Then the induced Hom-algebra $A_\alpha = (A,\mualpha=\alpha\circ\mu,\alpha)$ is a Hom-flexible algebra.
\item
Let $(A,\mu,\alpha)$ be a Hom-flexible algebra.  Then the derived Hom-algebra $A^n = (A,\mun,\alpha^{2^n})$ is also a Hom-flexible algebra for each $n \geq 0$, where $\mun = \alpha^{2^n-1}\circ\mu$.
\end{enumerate}
\end{theorem}

\begin{proof}
For the first assertion, for \emph{any} algebra $(A,\mu)$, we regard it as the Hom-algebra $(A,\mu,Id)$ with identity twisting map.  Then we have:
\begin{equation}
\label{associatoraalpha}
\begin{split}
as_{A_\alpha} &= \mualpha\circ(\mualpha\otimes\alpha - \alpha\otimes\mualpha)\\
&= \alpha^2 \circ \mu \circ (\mu \otimes Id - Id \otimes \mu)\quad\text{(by multiplicativity of $\alpha$)}\\
&= \alpha^2 \circ as_A.
\end{split}
\end{equation}
Now for a flexible algebra $(A,\mu)$, this implies that
\[
as_{A_\alpha}(x,y,x) = \alpha^2(as_A(x,y,x)) = 0,
\]
so $A_\alpha$ is Hom-flexible.

For the second assertion, we have that for \emph{any} Hom-algebra $(A,\mu,\alpha)$:
\begin{equation}
\label{associatoran}
\begin{split}
as_{A^n} &= \mun \circ (\mun \otimes \alpha^{2^n} - \alpha^{2^n} \otimes \mun)\\
&= \alpha^{2(2^n-1)} \circ \mu \circ (\mu \otimes \alpha - \alpha \otimes \mu)\quad\text{(by multiplicativity of $\alpha$)}\\
&= \alpha^{2(2^n-1)} \circ as_A.
\end{split}
\end{equation}
Now for a Hom-flexible algebra $(A,\mu,\alpha)$, this implies that
\[
as_{A^n}(x,y,x) = \alpha^{2(2^n-1)}(as_A(x,y,x)) = 0,
\]
so $A^n$ is Hom-flexible.
\end{proof}

\begin{theorem}
\label{thm:hommalad}
\begin{enumerate}
\item
Let $(A,\mu)$ be a Maltsev-admissible algebra and $\alpha \colon A \to A$ be an algebra morphism.  Then the induced Hom-algebra $A_\alpha = (A,\mualpha = \alpha\circ\mu,\alpha)$ is a Hom-Maltsev-admissible algebra.
\item
Let $(A,\mu,\alpha)$ be a Hom-Maltsev-admissible algebra.  Then the derived Hom-algebra $A^n = (A,\mun,\alpha^{2^n})$ is also a Hom-Maltsev-admissible algebra for each $n \geq 0$, where $\mun = \alpha^{2^n-1}\circ\mu$.
\end{enumerate}
\end{theorem}

\begin{proof}
For the first assertion, the commutator algebra of $(A,\mu)$ is $A^- = (A,[-,-] = \mu\circ(Id - \tau))$, which is a Maltsev algebra by assumption.  In particular, the Maltsev identity
\begin{equation}
\label{maltsevid}
J_{A^-}(x,y,[x,z]) = [J_{A^-}(x,y,z),x]
\end{equation}
holds.  The commutator Hom-algebra of the induced Hom-algebra $A_\alpha$ is $A_\alpha^- = (A,[-,-]_\alpha,\alpha)$, where
\begin{equation}
\label{bracketalpha}
[-,-]_\alpha = \mualpha\circ(Id-\tau) = \alpha \circ [-,-],
\end{equation}
which is anti-symmetric.  We must show that $A_\alpha^-$ satisfies the Hom-Maltsev identity \eqref{hommaltsevid}.  We have:
\begin{equation}
\label{jaalphaminus}
\begin{split}
J_{A_\alpha^-} &= [-,-]_\alpha \circ ([-,-]_\alpha \otimes \alpha) \circ (Id + \sigma + \sigma^2)\\
&= \alpha^2 \circ [-,-] \circ ([-,-] \otimes Id) \circ (Id + \sigma + \sigma^2) \quad\text{(by \eqref{bracketalpha})}\\
&= \alpha^2 \circ J_{A^-}.
\end{split}
\end{equation}
Therefore, we have:
\[
\begin{split}
J_{A_\alpha^-}(&\alpha(x),\alpha(y),[x,z]_\alpha) - [J_{A_\alpha^-}(x,y,z),\alpha^2(x)]_\alpha\\
&= \alpha^2\{J_{A^-}(\alpha(x),\alpha(y),\alpha[x,z])\} - \alpha[\alpha^2(J_{A^-}(x,y,z)),\alpha^2(z)]\quad\text{(by \eqref{jaalphaminus})}\\
&= \alpha^3\left\{J_{A^-}(x,y,[x,z]) - [J_{A^-}(x,y,z),x]\right\}\quad\text{(by \eqref{alphajj})}\\
&= 0\quad\text{(by \eqref{maltsevid})}.
\end{split}
\]
This shows that the Hom-Maltsev identity holds in $A_\alpha^-$, proving the first assertion.

For the second assertion, assume that $(A,\mu,\alpha)$ is a Hom-Maltsev-admissible algebra.  Note that the commutator Hom-algebra of the $n$th derived Hom-algebra $A^n$ is $(A^n)^- = (A,[-,-]^{(n)},\alpha^{2^n})$,
where
\[
[-,-]^{(n)} = \mun\circ(Id-\tau) = \alpha^{2^n-1} \circ [-,-].
\]
Thus, we have
\begin{equation}
\label{anminus}
(A^n)^- = (A^-)^n,
\end{equation}
where $A^- = (A,[-,-],\alpha)$ is the commutator Hom-algebra of $A$ and $(A^-)^n$ is its $n$th derived Hom-algebra (Definition \ref{def:derivedhomalge}).  Since $[-,-]^{(n)}$ is anti-symmetric, we must show that $(A^n)^-$ satisfies the Hom-Maltsev identity \eqref{hommaltsevid}.  To do that, observe that
\begin{equation}
\label{janminus}
\begin{split}
J_{(A^n)^-} &= J_{(A^-)^n} \quad\text{(by \eqref{anminus})}\\
&= \alpha^{2(2^n-1)}\circ J_{A^-}\quad\text{(by \eqref{jan})}.
\end{split}
\end{equation}
In the computation below, we write $k = 3(2^n-1)$.  Using \eqref{jalpha} and \eqref{janminus} we compute as follows:
\[
\begin{split}
J_{(A^n)^-}(&\alpha^{2^n}(x),\alpha^{2^n}(y),[x,z]^{(n)}) - [J_{(A^n)^-}(x,y,z),(\alpha^{2^n})^2(x)]^{(n)}\\
&= \alpha^{2(2^n-1)} \left\{J_{A^-}(\alpha^{2^n}(x),\alpha^{2^n}(y),\alpha^{2^n-1}[x,z])\right\}\\
&\relphantom{} - \alpha^{2^n-1}[\alpha^{2(2^n-1)}\circ J_{A^-}(x,y,z),\alpha^{2^{n+1}}(x)]\\
&= \alpha^k\left\{J_{A^-}(\alpha(x),\alpha(y),[x,z]) - [J_{A^-}(x,y,z),\alpha^2(x)]\right\}\\
&= 0.
\end{split}
\]
This last equality follows from the Hom-Maltsev identity \eqref{hommaltsevid} in $A^-$.  We have shown that $(A^n)^-$ satisfies the Hom-Maltsev identity, so $A^n$ is Hom-Maltsev-admissible.
\end{proof}


Below we consider examples of Hom-flexible, Hom-Maltsev-admissible algebras that are not Hom-alternative, not Hom-Lie-admissible, and not Maltsev-admissible.  In particular, these Hom-Maltsev-admissible algebras cannot be obtained from Theorem \ref{thm:homaltmaltsev}.  Therefore, the class of Hom-flexible, Hom-Maltsev-admissible algebras is strictly larger than the class of Hom-alternative algebras.

\begin{example}
\label{ex:hma5}
There is a five-dimensional flexible, Maltsev-admissible algebra $(A,\mu)$ (\cite{myung} Example 1.5, p.29) with basis $\{e_1,\ldots,e_5\}$ and multiplication table:
\begin{center}
\begin{tabular}{c|ccccc}
$\mu$ & $e_1$ & $e_2$ & $e_3$ & $e_4$ & $e_5$\\\hline
$e_1$ & $0$ & $e_5 + \frac{1}{2}e_4$ & $0$ & $\frac{1}{2}e_1$ & $0$\\
$e_2$ & $e_5 - \frac{1}{2}e_4$ & $0$ & $0$ & $-\frac{1}{2}e_2$ & $0$\\
$e_3$ & $0$ & $0$ & $0$ & $\frac{1}{2}e_3$ & $0$\\
$e_4$ & $-\frac{1}{2}e_1$ & $\frac{1}{2}e_2$ & $-\frac{1}{2}e_3$ & $-e_5$ & $0$\\
$e_5$ & $0$ & $0$ & $0$ &$0$ & $0$
\end{tabular}
\end{center}
This Maltsev-admissible algebra $A$ is neither alternative nor Lie-admissible.

Let $\lambda, \xi \in \bk$ be arbitrary scalars with $\lambda \not\in \{0,\pm 1\}$.  There is an algebra morphism $\alpha \colon A \to A$ given by
\[
\begin{split}
\alpha(e_1) &= \lambda e_1,\quad \alpha(e_2) = \lambda^{-1}e_2,\quad \alpha(e_3) = \xi e_3,\\
\alpha(e_4) &= e_4,\quad \alpha(e_5) = e_5.
\end{split}
\]
By Theorems \ref{thm:homf} and \ref{thm:hommalad} the induced Hom-algebra $A_\alpha = (A,\mualpha=\alpha\circ\mu,\alpha)$ is a Hom-flexible, Hom-Maltsev-admissible algebra.  Its multiplication table is:
\begin{center}
\begin{tabular}{c|ccccc}
$\mualpha$ & $e_1$ & $e_2$ & $e_3$ & $e_4$ & $e_5$\\\hline
$e_1$ & $0$ & $e_5 + \frac{1}{2}e_4$ & $0$ & $\frac{\lambda}{2}e_1$ & $0$\\
$e_2$ & $e_5 - \frac{1}{2}e_4$ & $0$ & $0$ & $-\frac{\lambda^{-1}}{2}e_2$ & $0$\\
$e_3$ & $0$ & $0$ & $0$ & $\frac{\xi}{2}e_3$ & $0$\\
$e_4$ & $-\frac{\lambda}{2}e_1$ & $\frac{\lambda^{-1}}{2}e_2$ & $-\frac{\xi}{2}e_3$ & $-e_5$ & $0$\\
$e_5$ & $0$ & $0$ & $0$ &$0$ & $0$
\end{tabular}
\end{center}
Note that $A_\alpha$ is not Hom-alternative because
\[
as_{A_\alpha}(e_1,e_2,e_2) = \frac{\lambda^{-2}}{4}e_2 \not= 0.
\]
Also, $A_\alpha$ is not Hom-Lie-admissible because
\[
J_{(A_\alpha)^-}(e_1,e_2,e_3) = -b^2e_3 \not= 0.
\]
Finally, $(A,\mualpha)$ is not Maltsev-admissible, i.e., $(A,[-,-]_\alpha)$ is not a Maltsev algebra, where $[-,-]_\alpha = \mualpha \circ (Id - \tau) = \alpha \circ [-,-]$.  Indeed, let $J'$ denote the usual Jacobian of $(A,[-,-]_\alpha)$ as in \eqref{j'}.  Then, on the one hand, we have
\[
\begin{split}
J'(e_1,e_2,[e_1,e_4]_\alpha) &=
[[e_1,e_2]_\alpha,[e_1,e_4]_\alpha]_\alpha + [[[e_1,e_4]_\alpha,e_1]_\alpha,e_2]_\alpha + [[e_2,[e_1,e_4]_\alpha]_\alpha,e_1]_\alpha\\
&= [e_4,\lambda e_1]_\alpha + [[\lambda e_1,e_1]_\alpha,e_2]_\alpha + [[e_2,\lambda e_1]_\alpha,e_1]_\alpha\\
&= -\lambda^2e_1 + 0 + \lambda^2e_1\\
&= 0.
\end{split}
\]
On the other hand, we have
\[
\begin{split}
[J'(e_1,e_2,e_4),e_1]_\alpha &=
[\cyclicsum [[e_1,e_2]_\alpha,e_4]_\alpha,e_1]_\alpha\\
&= [(\lambda^{-1}-\lambda)e_4,e_1]_\alpha\\
&= (\lambda^2-1)e_1,
\end{split}
\]
which is not equal to $0$ because $\lambda \not= \pm 1$.  So $(A,[-,-]_\alpha)$ does not satisfy the Maltsev identity \eqref{maltsevidentity} and, therefore, is not a Maltsev algebra.
\qed
\end{example}

\begin{example}
\label{ex:hma6}
There is a six-dimensional flexible, Maltsev-admissible algebra $(A,\mu)$ (\cite{myung} Table 5.10, p.301) with basis $\{e,h,f,u,v,w\}$ and multiplication table:
\begin{center}
\begin{tabular}{c|cccccc}
$\mu$ & $e$ & $h$ & $f$ & $u$ & $v$ & $w$\\\hline
$e$ & $0$ & $-e$ & $\frac{1}{2}h + \lambda u$ & $0$ & $w$ & $0$\\
$h$ & $e$ & $2\lambda u$ & $-f$ & $0$ & $v$ & $-w$\\
$f$ & $-\frac{1}{2}h+\lambda u$ & $f$ & $0$ & $0$ & $0$ & $v$\\
$u$ & $0$ & $0$ & $0$ & $0$ & $0$ & $0$\\
$v$ & $-w$ & $-v$ & $0$ & $0$ & $0$ & $\frac{1}{2}u$\\
$w$ & $0$ & $w$ & $-v$ & $0$ & $-\frac{1}{2}u$ & $0$
\end{tabular}
\end{center}
In the table above, $\lambda \in \bk$ is an arbitrary but fixed scalar.  This Maltsev-admissible algebra $A$ is neither alternative nor Lie-admissible.

Let $\gamma \in \bk$ be a scalar with $\gamma \not= 0$ and $\gamma^8 \not= 1$.  Then there is an algebra morphism $\alpha \colon A \to A$ given by
\[
\begin{split}
\alpha(e) &= \gamma^{-2}e,\quad \alpha(h) = h,\quad \alpha(f) = \gamma^2f,\\
\alpha(u) &= u, \quad \alpha(v) = \gamma v, \quad \alpha(w) = \gamma^{-1}w.
\end{split}
\]
By Theorems \ref{thm:homf} and \ref{thm:hommalad} the induced Hom-algebra $A_\alpha = (A,\mualpha=\alpha\circ\mu,\alpha)$ is a Hom-flexible, Hom-Maltsev-admissible algebra.  Its multiplication table is:
\begin{center}
\begin{tabular}{c|cccccc}
$\mualpha$ & $e$ & $h$ & $f$ & $u$ & $v$ & $w$\\\hline
$e$ & $0$ & $-\gamma^{-2}e$ & $\frac{1}{2}h + \lambda u$ & $0$ & $\gamma^{-1}w$ & $0$\\
$h$ & $\gamma^{-2}e$ & $2\lambda u$ & $-\gamma^2f$ & $0$ & $\gamma v$ & $-\gamma^{-1}w$\\
$f$ & $-\frac{1}{2}h+\lambda u$ & $\gamma^2f$ & $0$ & $0$ & $0$ & $\gamma v$\\
$u$ & $0$ & $0$ & $0$ & $0$ & $0$ & $0$\\
$v$ & $-\gamma^{-1}w$ & $-\gamma v$ & $0$ & $0$ & $0$ & $\frac{1}{2}u$\\
$w$ & $0$ & $\gamma^{-1}w$ & $-\gamma v$ & $0$ & $-\frac{1}{2}u$ & $0$
\end{tabular}
\end{center}
Note that $A_\alpha$ is not Hom-alternative because
\[
as_{A_\alpha}(h,h,f) = -\gamma^4f \not= 0.
\]
Also, $A_\alpha$ is not Hom-Lie-admissible because
\[
J_{(A_\alpha)^-}(h,f,w) = -12\gamma^2v \not= 0.
\]
Finally, $(A,\mualpha)$ is not Maltsev-admissible, i.e., $(A,[-,-]_\alpha)$ is not a Maltsev algebra, where $[-,-]_\alpha = \mualpha \circ (Id - \tau) = \alpha \circ [-,-]$.  Indeed, with $J'$ denoting the usual Jacobian of $(A_\alpha)^- = (A,[-,-]_\alpha)$ as in \eqref{j'}, we have
\[
\begin{split}
J'(e,h,[e,f]_\alpha) - [J'(e,h,f),e]_\alpha
&= 4(\gamma^{-4}-1)e - 4(\gamma^4-1)e\\
&= 4(\gamma^{-4}-\gamma^4)e.
\end{split}
\]
This is not equal to $0$ because $\gamma^8 \not= 1$.  So $(A,[-,-]_\alpha)$ does not satisfy the Maltsev identity \eqref{maltsevidentity} and, therefore, is not a Maltsev algebra.
\qed
\end{example}

\begin{example}
\label{ex:hma8}
There is an eight-dimensional flexible, Maltsev-admissible algebra $(A,\mu)$ \cite{myung} (Theorems 4.11 and 5.7) with basis $\{a,e_0,e_{\pm i} \colon i = 1,2,3\}$, whose multiplication is determined by:
\[
\begin{split}
e_0e_{\pm i} &= -e_{\pm i}e_0 = \pm e_{\pm i} \quad\text{for $i=1,2,3$},\\
e_{\pm i}e_{\pm j} &= -e_{\pm j}e_{\pm i} = \pm e_{\mp k} \quad\text{for $(ijk) = (123), (312), (231)$},\\
e_ie_{-i} &= \frac{1}{2}e_0 + \gamma a,\quad e_{-i}e_i = -\frac{1}{2}e_0 + \gamma a \quad\text{for $i=1,2,3$},\\
e_0^2 &= 2\gamma a,\quad a^2 = \delta a,\\
ax &= xa = \varepsilon x  \quad\text{for $x \in \{e_0,e_{\pm i}\}_{i=1,2,3}$}.\\
\end{split}
\]
The unspecified products of the basis elements are $0$, and $\gamma$, $\delta$, $\varepsilon$ are arbitrary but fixed scalars.  This Maltsev-admissible algebra $A$ is neither alternative nor Lie-admissible.

Let $\lambda,\xi \in \bk$ be non-zero scalars.  Then there is an algebra morphism $\alpha \colon A \to A$ given by
\[
\begin{split}
\alpha(a) &= a, \quad \alpha(e_0) = e_0,\\
\alpha(e_{\pm 1}) &= \lambda^{\pm 1}e_{\pm 1},\quad
\alpha(e_{\pm 2}) = \xi^{\pm 1}e_{\pm 2},\quad
\alpha(e_{\pm 3}) = (\lambda\xi)^{\mp 1}e_{\pm 3}.
\end{split}
\]
By Theorems \ref{thm:homf} and \ref{thm:hommalad} the induced Hom-algebra $A_\alpha = (A,\mualpha=\alpha\circ\mu,\alpha)$ is a Hom-flexible, Hom-Maltsev-admissible algebra.  Note that $A_\alpha$ is not Hom-alternative because
\[
as_{A_\alpha}(e_0,e_0,a) = 2\gamma(\delta - \varepsilon)a,
\]
which is not equal to $0$ in general.  Also, $A_\alpha$ is not Hom-Lie-admissible because
\[
J_{(A_\alpha)^-}(e_0,e_1,e_2) = 12(\lambda\xi)^2e_{-3} \not= 0.
\]
Finally, $(A,\mualpha)$ is not Maltsev-admissible, i.e., $(A,[-,-]_\alpha)$ is not a Maltsev algebra, where $[-,-]_\alpha = \mualpha \circ (Id - \tau) = \alpha \circ [-,-]$.  Indeed, with $J'$ denoting the usual Jacobian of $(A,[-,-]_\alpha)$ as in \eqref{j'}, we have
\[
J'(e_0,e_1,[e_0,e_2]_\alpha) - [J'(e_0,e_1,e_2),e_0]_\alpha = 8\gamma e_{-3},
\]
where
\[
\gamma = \lambda\xi^2(\lambda + \xi - 2\lambda\xi - \lambda^2 + \lambda^2\xi),
\]
which is not equal to $0$ in general.
\qed
\end{example}

\section{Hom-alternative algebras are Hom-Jordan-admissible}
\label{sec:jordan}

In this section, we define Hom-Jordan(-admissible) algebras.  Some alternative characterizations of the Hom-Jordan identity \eqref{homjordanid} are given in Proposition \ref{prop:homjordanchar}.  The main result of this section is Theorem \ref{thm:hahj}, which says that Hom-alternative algebras are Hom-Jordan-admissible.  Then we prove Theorems \ref{thm:hjtp} and \ref{thm:hjatp}, which are construction results for Hom-Jordan and Hom-Jordan-admissible algebras.  In Example \ref{ex:m83} we construct Hom-Jordan algebras from the $27$-dimensional exceptional simple Jordan algebra of $3 \times 3$ Hermitian octonionic matrices.

Let us begin with some relevant definitions.

\begin{definition}
\label{def:plushom}
Let $(A,\mu,\alpha)$ be a Hom-algebra.  Define its \textbf{plus Hom-algebra} as the Hom-algebra $A^+ = (A,\ast,\alpha)$, where $\ast = (\mu + \mu\circ\tau)/2$.
\end{definition}

With $\mu(x,y) = xy$, the product $\ast$ is given by
\[
x \ast y = \frac{1}{2}(\mu(x,y) + \mu(y,x)) = \frac{1}{2}(xy + yx),
\]
which is commutative.  Also, we have
\begin{equation}
\label{xsquare}
x \ast x = \mu(x,x) = x^2
\end{equation}
for $x \in A$.  In other words, $\mu$ and $\ast$ have the same squares.  In what follows, we will often abbreviate $x \ast x$ to $x^2$.

\begin{definition}
\label{def:homjordan}
\begin{enumerate}
\item
A \textbf{Hom-Jordan algebra} is a Hom-algebra $(A,\mu,\alpha)$ such that $\mu = \mu \circ \tau$ (commutativity) and the \textbf{Hom-Jordan identity}
\begin{equation}
\label{homjordanid}
as_A(x^2,\alpha(y),\alpha(x)) = 0
\end{equation}
is satisfied for all $x,y \in A$, where $as_A$ is the Hom-associator \eqref{homassociator}.
\item
A \textbf{Hom-Jordan-admissible algebra} is a Hom-algebra $(A,\mu,\alpha)$ whose plus Hom-algebra $A^+ = (A,\ast,\alpha)$ is a Hom-Jordan algebra.
\end{enumerate}
\end{definition}

The Hom-Jordan identity \eqref{homjordanid} can be rewritten as
\[
\mu(\mu(x^2,\alpha(y)),\alpha^2(x)) = \mu(\alpha(x^2),\mu(\alpha(y),\alpha(x))).
\]
Since the product $\ast$ is commutative, using \eqref{xsquare} a Hom-algebra $A$ is Hom-Jordan-admissible if and only if $A^+$ satisfies the Hom-Jordan identity
\begin{equation}
\label{homjordanaplus}
as_{A^+}(x^2,\alpha(y),\alpha(x)) = 0,
\end{equation}
or equivalently
\[
(x^2 \ast \alpha(y)) \ast \alpha^2(x) = \alpha(x^2) \ast (\alpha(y) \ast \alpha(x)),
\]
for all $x,y \in A$.

\begin{example}
A Jordan(-admissible) algebra is a Hom-Jordan(-admissible) algebra with $\alpha = Id$, since the Hom-Jordan identity \eqref{homjordanid} with $\alpha = Id$ is the Jordan identity \eqref{jordan}.  The reader is referred to \cite{albert,jacobson,schafer} for discussions about structures of Jordan algebras. Other ways of constructing Hom-Jordan(-admissible) algebras are given below.\qed
\end{example}

\begin{remark}
In \cite{mak} Makhlouf defined a Hom-Jordan algebra as a commutative Hom-algebra satisfying $as_A(x^2,y,\alpha(x)) = 0$, which becomes our Hom-Jordan identity \eqref{homjordanid} if $y$ is replaced by $\alpha(y)$.  This seemingly minor difference is, in fact, very significant with respect to Hom-Jordan-admissibility of Hom-alternative algebras.  Using Makhlouf's definition of a Hom-Jordan algebra, Hom-alternative algebras are not Hom-Jordan-admissible, although Hom-associative algebras are still Hom-Jordan-admissible \cite{mak} (Theorem 3.3).
\end{remark}

Let us give some alternative characterizations of the Hom-Jordan identity \eqref{homjordanid}, including a linearized version of it \eqref{homjordanid3}.  The ordinary (non-Hom) version of the following result can be found in, e.g., \cite{schafer} (Chapter IV).

\begin{proposition}
\label{prop:homjordanchar}
Let $(A,\mu,\alpha)$ be a Hom-algebra with $\mu$ commutative, i.e., $\mu = \mu\circ\tau$.  Then the following statements are equivalent:
\begin{enumerate}
\item
$A$ is a Hom-Jordan algebra, i.e., $A$ satisfies the Hom-Jordan identity \eqref{homjordanid}.
\item
$A$ satisfies
\begin{equation}
\label{homjordanid2}
2as_A(xz,\alpha(y),\alpha(x)) + as_A(x^2,\alpha(y),\alpha(z)) = 0
\end{equation}
for all $x,y,z \in A$.
\item
$A$ satisfies
\begin{equation}
\label{homjordanid3}
as_A(zx,\alpha(y),\alpha(w)) + as_A(wz,\alpha(y),\alpha(x)) + as_A(xw,\alpha(y),\alpha(z)) = 0
\end{equation}
for all $w,x,y,z \in A$.
\end{enumerate}
\end{proposition}

\begin{proof}
We will show the implications (1) $\Rightarrow$ (2) $\Rightarrow$ (3) $\Rightarrow$ (1).

First assume that $A$ is a Hom-Jordan algebra.  Replace $x$ with $x + \lambda z$ for $\lambda \in \bk$ in the Hom-Jordan identity \eqref{homjordanid}.  Using the commutativity of $\mu$ and \eqref{homjordanid}, the result is
\begin{equation}
\label{hj2}
\begin{split}
0 &= \lambda\left\{2as_A(xz,\alpha(y),\alpha(x)) + as_A(x^2,\alpha(y),\alpha(z))\right\}\\
&\relphantom{} + \lambda^2\left\{2as_A(xz,\alpha(y),\alpha(z)) + as_A(z^2,\alpha(y),\alpha(x))\right\}.
\end{split}
\end{equation}
Since \eqref{hj2} holds for both $\lambda = 1$ and $\lambda = -1$, it follows that the coefficient of $\lambda$ in \eqref{hj2} is equal to $0$, which is exactly the condition \eqref{homjordanid2}.  So statement (1) implies statement (2).

Next assume that $A$ satisfies \eqref{homjordanid2}.  Replace $x$ with $x + \gamma w$ for $\gamma \in \bk$ in \eqref{homjordanid2}.  By the same reasoning as in the previous paragraph, in the resulting expression the coefficient of $\gamma$ must be equal to $0$.  A simple computation shows that this coefficient of $\gamma$ is twice the left-hand side of \eqref{homjordanid3}.  Therefore, statement (2) implies statement (3).

Finally, starting from \eqref{homjordanid3}, one sets $w = z = x$ to obtain the Hom-Jordan identity \eqref{homjordanid}.
\end{proof}

Note that the linearized Hom-Jordan identity \eqref{homjordanid3} can be written as
\[
\cyclicsum_{x,w,z} \,as_A(xw,\alpha(y),\alpha(z)) = 0,
\]
where $\cyclicsum_{x,w,z}$ is the cyclic sum over $(x,w,z)$.

Here is the main result of this section.

\begin{theorem}
\label{thm:hahj}
Every Hom-alternative algebra is Hom-Jordan-admissible.
\end{theorem}

To prove Theorem \ref{thm:hahj}, we will use the following preliminary observation.

\begin{lemma}
\label{lemma:asaplus}
Let $(A,\mu,\alpha)$ be any Hom-algebra and $A^+ = (A,\ast,\alpha)$ be its plus Hom-algebra.  Then we have
\begin{equation}
\label{4as}
\begin{split}
4as_{A^+}(x^2,\alpha(y),\alpha(x))
&= as_A(x^2,\alpha(y),\alpha(x)) - as_A(\alpha(x),\alpha(y),x^2)\\
&\relphantom{} + as_A(\alpha(y),x^2,\alpha(x)) - as_A(\alpha(x),x^2,\alpha(y))\\
&\relphantom{} + as_A(x^2,\alpha(x),\alpha(y)) - as_A(\alpha(y),\alpha(x),x^2)\\
&\relphantom{} + [\alpha^2(y),as_A(x,x,x)]
\end{split}
\end{equation}
for all $x,y \in A$, where $[-,-] = \mu\circ(Id - \tau)$ is the commutator bracket.
\end{lemma}

\begin{proof}
As usual we write $\mu(a,b)$ as the juxtaposition $ab$, and $\mu(x,x) = x^2 = x \ast x$.  Starting from the left-hand side of \eqref{4as}, we have:
\begin{equation}
\label{4as1}
\begin{split}
4&as_{A^+}(x^2,\alpha(y),\alpha(x)) \\
&= 4(x^2 \ast \alpha(y)) \ast \alpha^2(x) - 4\alpha(x^2) \ast (\alpha(y) \ast \alpha(x))\\
&= (x^2\alpha(y))\alpha^2(x) + (\alpha(y)x^2)\alpha^2(x) + \alpha^2(x)(x^2\alpha(y)) + \alpha^2(x)(\alpha(y)x^2)\\
&\relphantom{} - \alpha(x^2)(\alpha(y)\alpha(x)) - \alpha(x^2)(\alpha(x)\alpha(y)) - (\alpha(y)\alpha(x))\alpha(x^2) - (\alpha(x)\alpha(y))\alpha(x^2)\\
&= as_A(x^2,\alpha(y),\alpha(x)) - as_A(\alpha(x),\alpha(y),x^2)\\
&\relphantom{} + (\alpha(y)x^2)\alpha^2(x) + \alpha^2(x)(x^2\alpha(y)) - \alpha(x^2)(\alpha(x)\alpha(y)) - (\alpha(y)\alpha(x))\alpha(x^2).
\end{split}
\end{equation}
Using the definition of the Hom-associator \eqref{homassociator}, the last four terms in \eqref{4as1} are:
\begin{equation}
\label{4as2}
\begin{split}
(\alpha(y)x^2)\alpha^2(x) &= as_A(\alpha(y),x^2,\alpha(x)) + \alpha^2(y)(x^2\alpha(x)),\\
\alpha^2(x)(x^2\alpha(y)) &= -as_A(\alpha(x),x^2,\alpha(y)) + (\alpha(x)x^2)\alpha^2(y),\\
- \alpha(x^2)(\alpha(x)\alpha(y)) &= as_A(x^2,\alpha(x),\alpha(y)) - (x^2\alpha(x))\alpha^2(y),\\
- (\alpha(y)\alpha(x))\alpha(x^2) &= - as_A(\alpha(y),\alpha(x),x^2) - \alpha^2(y)(\alpha(x)x^2).
\end{split}
\end{equation}
Note that
\begin{equation}
\label{4as3}
\begin{split}
[\alpha^2(y),as_A(x,x,x)] &= [\alpha^2(y), (x^2)\alpha(x) - \alpha(x)x^2]\\
&= \alpha^2(y)(x^2\alpha(x)) - \alpha^2(y)(\alpha(x)x^2) - (x^2\alpha(x))\alpha^2(y) + (\alpha(x)x^2)\alpha^2(y).
\end{split}
\end{equation}
The desired condition \eqref{4as} now follows from \eqref{4as1}, \eqref{4as2}, and \eqref{4as3}.
\end{proof}

\begin{proof}[Proof of Theorem \ref{thm:hahj}]
Let $(A,\mu,\alpha)$ be a Hom-alternative algebra.  To show that it is Hom-Jordan-admissible, it suffices to prove the Hom-Jordan identity for its plus Hom-algebra $A^+$ \eqref{homjordanaplus}.  To do this, first observe that $A$ itself satisfies the Hom-Jordan identity:
\[
\begin{split}
as_A(x^2,\alpha(y),\alpha(x)) &= \alpha^2(x)as_A(x,y,x) + as_A(x,y,x)\alpha^2(x)\quad\text{(by \eqref{homalt1})}\\
&= 0\quad\text{(by alternativity of $as_A$)}.
\end{split}
\]
Using again the alternativity of $as_A$, this implies that
\[
0 = (as_A \circ \theta)(x^2,\alpha(y),\alpha(x))
\]
for every permutation $\theta$ on three letters.  Since $as_A(x,x,x) = 0$ as well, it follows from Lemma \ref{lemma:asaplus} that
\[
4as_{A^+}(x^2,\alpha(y),\alpha(x)) = 0,
\]
from which the desired Hom-Jordan identity for $A^+$ \eqref{homjordanaplus} follows.
\end{proof}


The following construction results are the analogues of Theorems \ref{thm:maltsevtp2} and \ref{thm:maltsevtp1} for Hom-Jordan and Hom-Jordan-admissible algebras.

\begin{theorem}
\label{thm:hjtp}
\begin{enumerate}
\item
Let $(A,\mu)$ be a Jordan algebra and $\alpha \colon A \to A$ be an algebra morphism.  Then the induced Hom-algebra $A_\alpha = (A,\mualpha = \alpha \circ \mu,\alpha)$ is a Hom-Jordan algebra.
\item
Let $(A,\mu,\alpha)$ be a Hom-Jordan algebra.  Then the derived Hom-algebra $A^n = (A,\mun=\alpha^{2^n-1}\circ\mu,\alpha^{2^n})$ is also a Hom-Jordan algebra for each $n \geq 0$.
\end{enumerate}
\end{theorem}

\begin{proof}
For the first assertion, first note that $\mualpha = \alpha\circ\mu$ is commutative.   To prove the Hom-Jordan identity \eqref{homjordanid} in $A_\alpha$, regard $(A,\mu)$ as the Hom-algebra $(A,\mu,Id)$.  Then  we have:
\[
\begin{split}
as_{A_\alpha}(\mualpha(x,x),\alpha(y),\alpha(x))
&= as_{A_\alpha}(\alpha(x^2),\alpha(y),\alpha(x))\\
&= \alpha^2\left(as_A(\alpha(x^2),\alpha(y),\alpha(x))\right)\quad\text{(by \eqref{associatoraalpha})}\\
&= \alpha^3\left(as_A(x^2,y,x)\right)\\
&= 0\quad\text{(by \eqref{jordan} in $A$)}.
\end{split}
\]
This shows that $A_\alpha$ is a Hom-Jordan algebra.

For the second assertion, first note that $\mun = \alpha^{2^n-1}\circ\mu$ is commutative.  To prove the Hom-Jordan identity \eqref{homjordanid} in $A^n$, we compute as follows:
\[
\begin{split}
as_{A^n}(\mun(x,x),\alpha^{2^n}(y),\alpha^{2^n}(x))
&= \alpha^{2(2^n-1)} \circ as_A(\alpha^{2^n-1}(x^2),\alpha^{2^n}(y),\alpha^{2^n}(x))\quad\text{(by \eqref{associatoran})}\\
&= \alpha^{3(2^n-1)} \circ as_A(x^2,\alpha(y),\alpha(x))\\
&= 0 \quad\text{(by \eqref{homjordanid} in $A$)}.
\end{split}
\]
This shows that $A^n$ is a Hom-Jordan algebra.
\end{proof}

\begin{theorem}
\label{thm:hjatp}
\begin{enumerate}
\item
Let $(A,\mu)$ be a Jordan-admissible algebra and $\alpha \colon A \to A$ be an algebra morphism.  Then the induced Hom-algebra $A_\alpha = (A,\mualpha = \alpha \circ \mu,\alpha)$ is a Hom-Jordan-admissible algebra.
\item
Let $(A,\mu,\alpha)$ be a Hom-Jordan-admissible algebra.  Then the derived Hom-algebra $A^n = (A,\mun=\alpha^{2^n-1}\circ\mu,\alpha^{2^n})$ is also a Hom-Jordan-admissible algebra for each $n \geq 0$.
\end{enumerate}
\end{theorem}

\begin{proof}
For the first assertion, first note that the plus Hom-algebra $(A_\alpha)^+ = (A,\ast_\alpha,\alpha)$ satisfies
\[
\ast_\alpha = \frac{1}{2}(\mualpha + \mualpha \circ \tau) = \alpha \circ \frac{1}{2}(\mu + \mu\circ\tau) = \alpha \circ \ast.
\]
Therefore, we have $(A_\alpha)^+ = (A^+)_\alpha$, where $A^+$ is the Jordan-algebra $(A,\ast)$.  Since $\ast_\alpha$ is commutative, it remains to prove the Hom-Jordan identity in $(A_\alpha)^+ = (A^+)_\alpha$.  We compute as follows:
\[
\begin{split}
as_{(A^+)_\alpha}(\mualpha(x,x),\alpha(y),\alpha(x))
&= \alpha^2 \circ as_{A^+}(\alpha(x^2),\alpha(y),\alpha(x))\quad\text{(by \eqref{associatoraalpha} in $A^+$)}\\
&= \alpha^3\left(as_{A^+}(x^2,y,x)\right)\\
&= 0\quad\text{(by \eqref{jordan} in $A^+$)}.
\end{split}
\]
This shows that $(A_\alpha)^+$ satisfies the Hom-Jordan identity, so $A_\alpha$ is Hom-Jordan-admissible.

For the second assertion, first note that the plus Hom-algebra $(A^n)^+ = (A,\ast^{(n)},\alpha^{2^n})$ satisfies
\[
\ast^{(n)} = \frac{1}{2}(\mun + \mun \circ \tau) = \alpha^{2^n-1}\circ\frac{1}{2}(\mu + \mu \circ \tau) = \alpha^{2^n-1}\circ\ast.
\]
Therefore, we have $(A^n)^+ = (A^+)^n$, where $A^+$ is the Hom-Jordan algebra $(A,\ast,\alpha)$ and $(A^+)^n$ is its $n$th derived Hom-algebra.  Since $\ast^{(n)}$ is commutative, it remains to prove the Hom-Jordan identity in $(A^n)^+ = (A^+)^n$.  We compute as follows:
\[
\begin{split}
as_{(A^+)^n}&(\mun(x,x),\alpha^{2^n}(y),\alpha^{2^n}(x))\\
&= \alpha^{2(2^n-1)}\circ as_{A^+}(\alpha^{2^n-1}(x^2),\alpha^{2^n}(y),\alpha^{2^n}(x))\quad\text{(by \eqref{associatoran} in $A^+$)}\\
&= \alpha^{3(2^n-1)}\circ as_{A^+}(x^2,\alpha(y),\alpha(x))\\
&= 0 \quad\text{(by \eqref{homjordanid} in $A^+$)}.
\end{split}
\]
This shows that $(A^n)^+$ is a Hom-Jordan algebra, so $A^n$ is Hom-Jordan-admissible.
\end{proof}


\begin{example}
\label{ex:m83}
In this example, we discuss how (non-Jordan) Hom-Jordan algebras can be constructed from the $27$-dimensional exceptional simple Jordan algebra $M^8_3$.  First recall from Example \ref{ex:oct} the octonions $\oct$, which is an eight-dimensional alternative (but not associative) algebra with basis $\{e_0,\ldots,e_7\}$, where $e_0$ is a two-sided multiplicative unit.  For an octonion $x = \sum_{i=0}^7 b_ie_i$ with each $b_i \in \bk$, its \emph{conjugate} is defined as the octonion $\xbar = b_0e_0 - \sum_{i=1}^7 b_ie_i$.

The elements of $M^8_3$ are $3 \times 3$ Hermitian octonionic matrices, i.e., matrices of the form
\[
X =
\begin{pmatrix}
a_1 & x & y\\
\xbar & a_2 & z\\
\ybar & \zbar & a_3
\end{pmatrix}
\]
with each $a_i \in \bk$ and $x,y,z \in \oct$.  Here we are using the convention $a_i = a_ie_0$ for the diagonal elements.  This $\bk$-module $M^8_3$ becomes a Jordan algebra with the multiplication
\[
X \ast Y = \frac{1}{2}(XY + YX),
\]
where $XY$ and $YX$ are the usual matrix multiplication.  The reader is referred to  \cite{gt,jvw,okubo,schafer} for discussions about the Jordan algebra $M^8_3$ and its relationship with the exceptional Lie algebra $F_4$.

Let $\alpha \colon \oct \to \oct$ be any unit-preserving and conjugate-preserving algebra morphism, i.e., $\alpha(e_0) = e_0$ and $\alpha(\xbar) = \overline{\alpha(x)}$ for all $x \in \oct$.  Then it extends entrywise to a linear map $\alpha \colon M^8_3 \to M^8_3$.  It is easy to see that this extended map $\alpha$ respects matrix multiplication and hence also the Jordan product $\ast$, i.e., $\alpha$ is an algebra morphism on $(M^8_3,\ast)$.  By the first part of Theorem \ref{thm:hjtp}, the induced Hom-algebra
\[
(M^8_3)_\alpha = (M^8_3,\ast_\alpha = \alpha\circ\ast,\alpha)
\]
is a Hom-Jordan algebra.

Note that $(M^8_3,\ast_\alpha = \alpha\circ\ast)$ is in general not a Jordan algebra.  For instance, consider the algebra automorphism $\alpha \colon \oct \to \oct$ defined in \eqref{octaut}, which is both unit-preserving and conjugate-preserving.  We claim that $(M^8_3,\ast_\alpha)$ is not a Jordan algebra, i.e., the Jordan identity
\begin{equation}
\label{jordanM}
((X \ast_\alpha X) \ast_\alpha Y) \ast_\alpha X = (X \ast_\alpha X) \ast_\alpha (Y \ast_\alpha X)
\end{equation}
is not satisfied for some $X,Y \in M^8_3$.  Indeed, using the multiplicativity of $\alpha$ with respect to $\ast$, the left-hand side in \eqref{jordanM} is
\[
\alpha(\alpha(\alpha(X^2) \ast Y) \ast X),
\]
where $X^2 = X \ast X$, and its right-hand side is
\[
\alpha^2(X^2 \ast (Y \ast X)).
\]
Since $\alpha$ is invertible, to show that $(M^8_3,\ast_\alpha)$ is not a Jordan algebra, it suffices to exhibit two elements $X, Y \in M^8_3$ such that
\begin{equation}
\label{notjordan}
\alpha(\alpha(X^2) \ast Y) \ast X \not= \alpha(X^2 \ast (Y \ast X)).
\end{equation}
Now we pick
\[
X = \matrixx \quad\text{and}\quad
Y = \matrixy
\]
in $M^8_3$.  With a little bit of computation, we obtain
\[
\alpha(\alpha(X^2) \ast Y) \ast X = \matrixz
\]
and
\[
\alpha(X^2 \ast (Y \ast X)) = \matrixw.
\]
Therefore, we have proved \eqref{notjordan}, so $(M^8_3,\ast_\alpha)$ is not a Jordan algebra.\qed
\end{example}

\section{Further properties of Hom-alternative algebras}
\label{sec:moufang}

In this section, we consider further properties of Hom-alternative algebras, including Hom-type analogues of the Moufang identities \cite{moufang} (Theorem \ref{thm:moufang}) and more identities concerning the Hom-Bruck-Kleinfeld function \eqref{f} (Propositions \ref{prop:f2} and \ref{prop:g}).

By Theorem \ref{thm:homaltmaltsev} every Hom-alternative algebra is Hom-Maltsev-admissible.  As we saw in Example \ref{ex:oct}, there are Hom-alternative algebras that are not Hom-Lie-admissible.  It is, therefore,  natural to ask which Hom-alternative algebras are Hom-Lie-admissible.
The following result says that the intersection of the classes of Hom-alternative algebras and of Hom-Lie-admissible algebras (within the class of Hom-algebras) is precisely the class of Hom-associative algebras.

\begin{proposition}
\label{cor:homaltass}
Let $(A,\mu,\alpha)$ be a Hom-algebra.  Then $A$ is a Hom-associative algebra if and only if it is both a Hom-alternative algebra and a Hom-Lie-admissible algebra.
\end{proposition}

\begin{proof}
If $A$ is Hom-associative, then by definition $as_A = 0$, which is alternating, so $A$ is Hom-alternative \cite{mak}.  It is observed in \cite{ms} that Hom-associative algebras are always Hom-Lie-admissible.  Conversely, if $A$ is Hom-alternative and Hom-Lie-admissible, then
\[
\begin{split}
as_A &= \frac{1}{6}J_{A^-}\quad\text{(by Proposition \ref{prop2:homalt})}\\
&= 0 \quad\text{($A^-$ is Hom-Lie)}.
\end{split}
\]
So $A$ is Hom-associative.
\end{proof}

It is proved in \cite{mak} that there is an analogue of Theorem \ref{thm:maltsevtp1} for Hom-alternative algebras.  It is a variation of \cite{yau2} (Theorem 2.3), which deals with $G$-Hom-associative algebras.  More precisely, if $(A,\mu)$ is an alternative algebra and $\alpha \colon A \to A$ is an algebra morphism, then the induced Hom-algebra $A_\alpha = (A,\mualpha = \alpha\circ\mu,\alpha)$ is a Hom-alternative algebra.

The following result is the analogue of Theorem \ref{thm:maltsevtp2} for Hom-alternative algebras.  It says that the category of Hom-alternative algebras is closed under taking derived Hom-algebras (Definition \ref{def:derivedhomalge}).

\begin{proposition}
\label{prop:alttp2}
Let $(A,\mu,\alpha)$ be a Hom-alternative algebra.  Then the $n$th derived Hom-algebra $A^n = (A,\mun = \alpha^{2^n-1}\circ\mu,\alpha^{2^n})$ is also a Hom-alternative algebra for each $n \geq 0$.
\end{proposition}

\begin{proof}
Indeed, $as_{A^n}$ is alternating because
\[
as_{A^n} = \alpha^{2(2^n-1)} \circ as_A
\]
by \eqref{associatoran} and $as_A$ is alternating.
\end{proof}

Next we provide further properties of the Hom-Bruck-Kleinfeld function $f$ \eqref{f}.  The following result gives two characterizations of the Hom-Bruck-Kleinfeld function in a Hom-alternative algebra.  It is the Hom-type analogue of part of \cite{bk} (Lemma 2.1).

\begin{proposition}
\label{prop:f2}
Let $(A,\mu,\alpha)$ be a Hom-alternative algebra.  Then the Hom-Bruck-Kleinfeld function $f$ satisfies
\begin{equation}
\label{f'}
f = \frac{1}{3}F = as \circ ([-,-]\otimes \alpha^{\otimes 2}) \circ (Id + \zeta),
\end{equation}
where $F$ is defined in \eqref{F} and $\zeta$ is the permutation $\zeta(w \otimes x \otimes y \otimes z) = y \otimes z \otimes w \otimes x$
\end{proposition}

\begin{proof}
First note that
\[
f = -f\circ\rho = f\circ\rho^2
\]
because $f$ is alternating (Proposition \ref{prop:f}), where $\rho = \xi^3$ is the cyclic permutation $\rho(w \otimes x \otimes y \otimes z) = x \otimes y \otimes z \otimes w$.  Therefore, we have
\[
\begin{split}
F &= f \circ (Id - \rho + \rho^2) \quad\text{(by Lemma \ref{lem2:homalt})}\\
&= 3f,
\end{split}
\]
which proves the first equality in \eqref{f'}.  It remains to prove that $f$ is equal to the last entry in \eqref{f'}.

Since $f$ is alternating, from its definition \eqref{f} we have
\[
\begin{split}
2f(w,x,y,z) &= f(w,x,y,z) - f(x,w,y,z)\\
&= as(wx,\alpha(y),\alpha(z)) - as(x,y,z)\alpha^2(w) - \alpha^2(x)as(w,y,z)\\
&\relphantom{} - as(xw,\alpha(y),\alpha(z)) + as(w,y,z)\alpha^2(x) + \alpha^2(w)as(x,y,z).
\end{split}
\]
Rearranging terms we obtain
\begin{equation}
\label{asbracket}
\begin{split}
as([w,x],\alpha(y),\alpha(z)) &= [\alpha^2(x),as(w,y,z)] - [\alpha^2(w),as(x,y,z)] + 2f(w,x,y,z)\\
&= [\alpha^2(x),as(y,z,w)] - [\alpha^2(w),as(x,y,z)] + 2f(w,x,y,z),
\end{split}
\end{equation}
in which $as(w,y,z) = as(y,z,w)$ because $as$ is alternating.  Interchanging $(w,x)$ with $(y,z)$ in \eqref{asbracket} and using the alternativity of $f$, we obtain
\begin{equation}
\label{asbracket'}
\begin{split}
as([y,z],\alpha(w),\alpha(x))
&= [\alpha^2(z),as(w,x,y)] - [\alpha^2(y),as(z,w,x)] + 2f(y,z,w,x)\\
&= [\alpha^2(z),as(w,x,y)] - [\alpha^2(y),as(z,w,x)] + 2f(w,x,y,z).
\end{split}
\end{equation}
Adding \eqref{asbracket} and \eqref{asbracket'} we have
\[
\begin{split}
as \circ(&[-,-]\otimes \alpha^{\otimes 2}) \circ (Id + \zeta)(w \otimes x \otimes y \otimes z)\\
&= as([w,x],\alpha(y),\alpha(z)) + as([y,z],\alpha(w),\alpha(x))\\
&= (4f - F)(w,x,y,z)\quad\text{(by \eqref{F'})}\\
&= f(w,x,y,z),
\end{split}
\]
since $F = 3f$.  This proves that $f$ is equal to the last entry in \eqref{f'}.
\end{proof}


In Proposition \ref{prop:f} we showed that the Hom-Bruck-Kleinfeld function $f$ \eqref{f} in a Hom-alternative algebra is an alternating function on four variables.  We now discuss a closely related  function on five variables in a Hom-alternative algebra.

\begin{definition}
\label{def:g}
Let $(A,\mu,\alpha)$ be a Hom-algebra.  Define the multi-linear map $g \colon A^{\otimes 5} \to A$ by
\begin{equation}
\label{g}
\begin{split}
g(u,v,w,x,y) &= f(uv,\alpha(w),\alpha(x),\alpha(y)) - \alpha^3(u)f(v,w,x,y)\\
&\relphantom{} - f(u,w,x,y)\alpha^3(v) - \alpha\left(as(u,x,y)\alpha[v,w]\right)\\
&\relphantom{} - \alpha\left((\alpha[u,w])as(v,x,y)\right)
\end{split}
\end{equation}
for $u,v,w,x,y \in A$, where $f$ is the Hom-Bruck-Kleinfeld function \eqref{f} and $[-,-] = \mu\circ(Id - \tau)$ is the commutator bracket of $\mu$.
\end{definition}

The following result says that the map $g$ is almost alternating in a Hom-alternative algebra and is the Hom-type analogue of \cite{bk} (Lemma 2.3).  Since $g$ is constructed using $\alpha$, $\mu$, $[-,-]$, $as$, and $f$ (which is defined using $\alpha$, $\mu$, and $as$), the following result is ultimately about identities in a Hom-alternative algebra.

\begin{proposition}
\label{prop:g}
In a Hom-alternative algebra $(A,\mu,\alpha)$, the map $g$ \eqref{g} is alternating in $\{u,v,w\}$ and also in $\{x,y\}$.  That is, $g$ changes sign when two of $\{u,v,w\}$ (or $x$ and $y$) are interchanged.
\end{proposition}

\begin{proof}
The map $g$ is alternating in $\{x,y\}$ because $f$ and $as$ are both alternating (the former by Proposition \ref{prop:f}).

To show that $g$ is alternating in $\{u,v,w\}$, first note that it is enough to show that $g$ is alternating in $\{u,w\}$ and in $\{v,w\}$.  The map $g$ is alternating in $\{u,w\}$ if and only if
\begin{equation}
\label{guvu}
g(u,v,u,x,y) = 0.
\end{equation}
Since $f$ is alternating and $[u,u] = 0$, \eqref{guvu} is equivalent to
\begin{equation}
\label{galt1}
f(uv,\alpha(u),\alpha(x),\alpha(y)) = \alpha^3(u)f(v,u,x,y) + \alpha\left(as(u,x,y)\alpha[v,u]\right).
\end{equation}
Likewise, $g$ is alternating in $\{v,w\}$ if and only if
\[
g(u,v,v,x,y) = 0,
\]
which is equivalent to
\begin{equation}
\label{galt2}
f(uv,\alpha(v),\alpha(x),\alpha(y)) = f(u,v,x,y)\alpha^3(v) + \alpha\left((\alpha[u,v])as(v,x,y)\right).
\end{equation}
It remains to prove \eqref{galt1} and \eqref{galt2}, which we do in the following two Lemmas.
\end{proof}

\begin{lemma}
\label{lem:galt1}
In a Hom-alternative algebra $(A,\mu,\alpha)$, \eqref{galt1} holds.
\end{lemma}

\begin{proof}
To prove \eqref{galt1}, we start with
\begin{equation}
\label{fvu}
f(vu,\alpha(u),\alpha(x),\alpha(y))
= as([vu,\alpha(u)],\alpha^2(x),\alpha^2(y)) + as([\alpha(x),\alpha(y)],\alpha(vu),\alpha^2(u)),
\end{equation}
which follows from Proposition \ref{prop:f2}.  We have
\[
\begin{split}
[vu,\alpha(u)] &= (vu)\alpha(u) - \alpha(u)(vu)\\
&= (vu)\alpha(u) - (uv)\alpha(u)\\
&= [v,u]\alpha(u).
\end{split}
\]
Therefore, using the definition \eqref{f} of $f$, the first summand on the right-hand side of \eqref{fvu} is:
\begin{equation}
\label{fvu1}
\begin{split}
as([vu,\alpha(u)],\alpha^2(x),\alpha^2(y))
&= as([v,u]\alpha(u),\alpha^2(x),\alpha^2(y))\\
&=f([v,u],\alpha(u),\alpha(x),\alpha(y)) + as(\alpha(u),\alpha(x),\alpha(y))\alpha^2([v,u])\\
&\relphantom{} + \alpha^2(\alpha(u))as([v,u],\alpha(x),\alpha(y))\\
&= f(vu,\alpha(u),\alpha(x),\alpha(y)) - f(uv,\alpha(u),\alpha(x),\alpha(y)) \\
&\relphantom{} + \alpha(as(u,x,y)\alpha[v,u]) + \alpha^3(u)as([v,u],\alpha(x),\alpha(y)).
\end{split}
\end{equation}
In the last equality above, we used the multiplicativity of $\alpha$.  On the other hand, the second summand on the right-hand side of \eqref{fvu} is:
\begin{equation}
\label{fvu2}
\begin{split}
as([\alpha(x),\alpha(y)],\alpha(vu),\alpha^2(u))
&= as(\alpha[x,y],\alpha(v)\alpha(u),\alpha^2(u))\quad\text{(by multiplicativity of $\alpha$)}\\
&= - as(\alpha^2(u),\alpha(v)\alpha(u),\alpha[x,y])\quad\text{(by alternativity of $as$)}\\
&= - \alpha^2(\alpha(u))as(\alpha(u),\alpha(v),[x,y])\quad\text{(by \eqref{homalt3})}\\
&= \alpha^3(u)as([x,y],\alpha(v),\alpha(u))\quad\text{(by alternativity of $as$)}.
\end{split}
\end{equation}
Using \eqref{fvu1} and \eqref{fvu2} in \eqref{fvu}, we obtain:
\[
\begin{split}
f(uv,\alpha(u),\alpha(x),\alpha(y))
&= \alpha(as(u,x,y)\alpha[v,u])\\
&\relphantom{} + \alpha^3(u)\{as([v,u],\alpha(x),\alpha(y)) + as([x,y],\alpha(v),\alpha(u))\}\\
&= \alpha(as(u,x,y)\alpha[v,u]) + \alpha^3(u)f(v,u,x,y),
\end{split}
\]
where the last equality follows from Proposition \ref{prop:f2} again.  This finishes the proof.
\end{proof}

\begin{lemma}
\label{lem:galt2}
In a Hom-alternative algebra $(A,\mu,\alpha)$, \eqref{galt2} holds.
\end{lemma}

\begin{proof}
To prove \eqref{galt2}, we start with
\begin{equation}
\label{fv}
f(vu,\alpha(v),\alpha(x),\alpha(y)) = as([vu,\alpha(v)],\alpha^2(x),\alpha^2(y)) + as([\alpha(x),\alpha(y)],\alpha(vu),\alpha^2(v)),
\end{equation}
which follows from Proposition \ref{prop:f2}.  We have
\[
\begin{split}
[vu,\alpha(v)] &= (vu)\alpha(v) - \alpha(v)(vu)\\
&= \alpha(v)(uv) - \alpha(v)(vu)\\
&= \alpha(v)[u,v].
\end{split}
\]
Therefore, using the definition \eqref{f} of $f$, the first summand on the right-hand side of \eqref{fv} is:
\begin{equation}
\label{fv1}
\begin{split}
as([vu,\alpha(v)],\alpha^2(x),\alpha^2(y))
&= as(\alpha(v)[u,v],\alpha^2(x),\alpha^2(y))\\
&= f(\alpha(v),[u,v],\alpha(x),\alpha(y)) + as([u,v],\alpha(x),\alpha(y))\alpha^3(v)\\
&\relphantom{} + \alpha^2([u,v])as(\alpha(v),\alpha(x),\alpha(y))\\
&= f(vu,\alpha(v),\alpha(x),\alpha(y)) - f(uv,\alpha(v),\alpha(x),\alpha(y))\\
&\relphantom{} + as([u,v],\alpha(x),\alpha(y))\alpha^3(v) + \alpha\left((\alpha[u,v])as(v,x,y)\right).
\end{split}
\end{equation}
In the last equality above, we used the alternativity of $f$ (Proposition \ref{prop:f}) and the multiplicativity of $\alpha$.  On the other hand, the second summand on the right-hand side of \eqref{fv} is:
\begin{equation}
\label{fv2}
\begin{split}
as([\alpha(x),\alpha(y)],\alpha(vu),\alpha^2(v))
&= as(\alpha[x,y],\alpha(v)\alpha(u),\alpha^2(v))\quad\text{(by multiplicativity of $\alpha$)}\\
&= -as(\alpha^2(v),\alpha(v)\alpha(u),\alpha[x,y])\quad\text{(by alternativity of $as$)}\\
&= -as(\alpha(v),\alpha(u),[x,y])\alpha^3(v)\quad\text{(by \eqref{homalt2})}\\
&= as([x,y],\alpha(u),\alpha(v))\alpha^3(v)\quad\text{(by alternativity of $as$)}.
\end{split}
\end{equation}
Using \eqref{fv1} and \eqref{fv2} in \eqref{fv}, we obtain:
\[
\begin{split}
f(uv,\alpha(v),\alpha(x),\alpha(y))
&= \alpha\left((\alpha[u,v])as(v,x,y)\right)\\
&\relphantom{} + \{as([u,v],\alpha(x),\alpha(y)) + as([x,y],\alpha(u),\alpha(v))\}\alpha^3(v)\\
&= \alpha\left((\alpha[u,v])as(v,x,y)\right) + f(u,v,x,y)\alpha^3(v),
\end{split}
\]
where the last equality follows from Proposition \ref{prop:f2} again.  This finishes the proof.
\end{proof}


In any alternative algebra, the following \textbf{Moufang identities} \cite{moufang} hold:
\begin{equation}
\label{moufang}
\begin{split}
(xyx)z &= x(y(xz)),\\
((zx)y)x &= z(xyx),\\
(xy)(zx) &= x(yz)x.
\end{split}
\end{equation}
Here $xyx = (xy)x = x(yx)$ is unambiguous in an alternative algebra.  Now we prove analogues of the Moufang identities in a Hom-alternative algebra.  The proof below is the Hom version of that of \cite{bk} (Lemma 2.2).

\begin{theorem}[\textbf{Hom-Moufang identities}]
\label{thm:moufang}
Let $(A,\mu,\alpha)$ be a Hom-alternative algebra.  Then the following Hom-Moufang identities hold for all $x,y,z \in A$:
\begin{subequations}
\label{hommoufang}
\begin{align}
((xy)\alpha(x))\alpha^2(z) &= \alpha^2(x)(\alpha(y)(xz)),\label{hm1}\\
((zx)\alpha(y))\alpha^2(x) &= \alpha^2(z)(\alpha(x)(yx)),\label{hm2}\\
\alpha((xy)(zx)) &= \alpha^2(x)((yz)\alpha(x)).\label{hm3}
\end{align}
\end{subequations}
\end{theorem}

\begin{proof}
For \eqref{hm1} we compute as follows:
\[
\begin{split}
((xy)\alpha(x))\alpha^2(z) &= as(xy,\alpha(x),\alpha(z)) + \alpha(xy)(\alpha(x)\alpha(z)) \quad\text{(by \eqref{homassociator})}\\
&= as(xy,\alpha(x),\alpha(z)) + (\alpha(x)\alpha(y))\alpha(xz) \quad\text{(by multiplicativity of $\alpha$)}\\
&= as(xy,\alpha(x),\alpha(z)) + as(\alpha(x),\alpha(y),xz) + \alpha^2(x)(\alpha(y)(xz))\quad\text{(by \eqref{homassociator})}\\
&= -as(\alpha(x),\alpha(y),xz) + as(\alpha(x),\alpha(y),xz) + \alpha^2(x)(\alpha(y)(xz))\quad\text{(by \eqref{homalt2})}\\
&= \alpha^2(x)(\alpha(y)(xz)).
\end{split}
\]
For \eqref{hm2} we compute as follows:
\[
\begin{split}
((zx)\alpha(y))\alpha^2(x) &= as(zx,\alpha(y),\alpha(x)) + \alpha(zx)(\alpha(y)\alpha(x))\quad\text{(by \eqref{homassociator})}\\
&= as(zx,\alpha(y),\alpha(x)) + (\alpha(z)\alpha(x))\alpha(yx)\quad\text{(by multiplicativity of $\alpha$)}\\
&= as(zx,\alpha(y),\alpha(x)) + as(\alpha(z),\alpha(x),yx) + \alpha^2(z)(\alpha(x)(yx))\quad\text{(by \eqref{homassociator})}\\
&= -as(\alpha(z),\alpha(x),yx)) + as(\alpha(z),\alpha(x),yx) + \alpha^2(z)(\alpha(x)(yx))\quad\text{(by \eqref{homalt3})}\\
&= \alpha^2(z)(\alpha(x)(yx)).
\end{split}
\]
For \eqref{hm3} we compute as follows:
\[
\begin{split}
\alpha((xy)(zx)) &= (\alpha(x)\alpha(y))\alpha(zx)\quad\text{(by multiplicativity of $\alpha$)}\\
&= as(\alpha(x),\alpha(y),zx) + \alpha^2(x)(\alpha(y)(zx))\quad\text{(by \eqref{homassociator})}\\
&= \alpha^2(x)as(x,y,z) + \alpha^2(x)(\alpha(y)(zx))\quad\text{(by \eqref{homalt3})}\\
&= \alpha^2(x)\left\{as(y,z,x) + \alpha(y)(zx)\right\}\quad\text{(by alternativity of $as$)}\\
&= \alpha^2(x)((yz)\alpha(x)).
\end{split}
\]
This completes the proof.
\end{proof}


\end{document}